\documentclass{article}
\usepackage{amsmath}

\newtheorem{theorem}{Theorem}

\newtheorem{conjecture}[theorem]{Conjecture}

\newtheorem{definition}[theorem]{Definition}

\newtheorem{lemma}[theorem]{Lemma}

\newtheorem{remark}[theorem]{Remark}

\newenvironment{proof}[1][Proof]{\noindent\textbf{#1.} }{\ \rule{0.5em}{0.5em}}
\input{tcilatex}

\begin{document}

\title{Means and Hermite Interpolation}
\author{Alan Horwitz}
\maketitle

\begin{abstract}
Let $m_{2}<m_{1}$ be two given nonnegative integers with $n=m_{1}+m_{2}+1$.
For suitably differentiable $f$, we let $P,Q\in \pi _{n}$ be the Hermite
polynomial interpolants to $f$ which satisfy $%
P^{(j)}(a)=f^{(j)}(a),j=0,1,...,m_{1}$ and $%
P^{(j)}(b)=f^{(j)}(b),j=0,1,...,m_{2},$ $%
Q^{(j)}(a)=f^{(j)}(a),j=0,1,...,m_{2}$ and $%
Q^{(j)}(b)=f^{(j)}(b),j=0,1,...,m_{1}$. Suppose that $f\in C^{n+2}(I)$ with $%
f^{(n+1)}(x)\neq 0$ for $x\in (a,b)$. If $m_{1}-m_{2}$ is even, then there
is a unique $x_{0},a<x_{0}<b,$ such that $P(x_{0})=Q(x_{0})$. If $%
m_{1}-m_{2} $ is odd, then there is a unique $x_{0},a<x_{0}<b,$ such that $%
f(x_{0})=\tfrac{1}{2}\left( P(x_{0})+Q(x_{0})\right) $. $x_{0}$ defines a
strict, symmetric mean, which we denote by $M_{f,m_{1},m_{2}}(a,b)$. We
prove various properties of these means. In particular, we show that $%
f(x)=x^{m_{1}+m_{2}+2}$ yields the arithmetic mean, $f(x)=x^{-1}$ yields the
harmonic mean, and $f(x)=x^{(m_{1}+m_{2}+1)/2}$ yields the geometric mean.
\end{abstract}

\section{Introduction}

\begin{definition}
A \textit{mean} $m(a,b)$ in two variables is a continuous function on $\Re
_{2}^{+}=\left\{ (a,b):a,b>0\right\} $ with $\min (a,b)\leq m(a,b)\leq \max
(a,b)$. $m$ is called

(1) \textit{Strict} if $m(a,b)=\min (a,b)$ or $m(a,b)=\max (a,b)$ if and
only if $a=b$ for all $(a,b)\in \Re _{2}^{+}$.

(2) \textit{Symmetric} if $m(b,a)=m(a,b)\ $for all$\ (a,b)\in \Re _{2}^{+}$.

(3) \textit{Homogeneous} if $m(ka,kb)=km(a,b)$ for any $k>0\ $and for all $%
(a,b)\in \Re _{2}^{+}$.
\end{definition}

Of course, in some cases a mean can be extended to all real numbers, such as
with the arithmetic mean $m(a,b)=\tfrac{a+b}{2}$. In this paper we define
means in two variables using intersections of Hermite polynomial
interpolants to a given function, $f$. Throughout we assume, unless stated
otherwise, that $m_{2}<m_{1}$ are two given nonnegative integers with $%
n=m_{1}+m_{2}+1$. If $f^{(k)}(a)$ and $f^{(k)}(b)$ each exist for $%
k=0,1,...,m_{1}$, we let $P,Q\in \pi _{n}$ be the Hermite polynomial
interpolants to $f$ which satisfy 
\begin{eqnarray}
P^{(j)}(a) &=&f^{(j)}(a),j=0,1,...,m_{1}\text{ and }%
P^{(j)}(b)=f^{(j)}(b),j=0,1,...,m_{2}  \notag \\
&&  \label{PQ} \\
Q^{(j)}(a) &=&f^{(j)}(a),j=0,1,...,m_{2}\text{ and }%
Q^{(j)}(b)=f^{(j)}(b),j=0,1,...,m_{1}  \notag
\end{eqnarray}

Of course $P$ and $Q$ depend on $m_{1},m_{2},$ and $f$, but we supress that
in our notation. Under suitable conditions on $f$(see Theorem \ref{T1}
below), if $m_{1}-m_{2}$ is even, then there is a unique $x_{0},a<x_{0}<b,$
such that $P(x_{0})=Q(x_{0})$. If $m_{1}-m_{2}$ is odd (see Theorem \ref{T2}
below), then there is a unique $x_{0},a<x_{0}<b,$ such that $f(x_{0})=\tfrac{%
1}{2}\left( P(x_{0})+Q(x_{0})\right) $. In either case, $x_{0}$ defines a
strict, symmetric mean, which we denote by $M_{f,m_{1},m_{2}}(a,b)$.

\qquad The means defined in this paper are similar to a class of means
defined in \cite{h1} and \cite{h2}, which were based on intersections of
Taylor polynomials, each of order $r$. More precisely, for $f\in
C^{r+1}(I),I=(a,b),$ let $P_{c}$ denote the Taylor polynomial to $f$ of
order $r$ at $x=c$, where $r$ is odd. In \cite{h1} it was proved that if $%
f^{(r+1)}(x)\neq 0$ on $[a,b]$, then there is a unique $u,a<u<b,$ such that $%
P_{a}(u)=P_{b}(u)$. This defines a mean $m(a,b)\equiv u$. These means were
extended to the case when $r$ is even in \cite{h2} by defining $m(a,b)$ to
be the unique solution in $(a,b)$ of the equation $f(x)=\dfrac{1}{2}\left(
P(x)+Q(x)\right) $. However, many of the proofs in this paper are more
complex than those in \cite{h1} and \cite{h2} because the means $%
M_{f,m_{1},m_{2}}$ depend on two nonnegative integers, $m_{1}$ and $m_{2}$,
rather than just on the one nonnegative integer, $r$. In \cite{h1} the
author also proved some minimal results for means involving intersections of
Hermite interpolants to a given function, $f$. In particular we proved a
version of Theorems \ref{T1}, \ref{T3}, and \ref{T4} below for the special
case when $m_{1}-m_{2}=2$. In this paper we prove much more along these
lines.

\section{Main Results}

Our first result allows us to define a mean using intersections of Hermite
interpolants when $m_{1}-m_{2}$ is even.

\begin{theorem}
\label{T1} Suppose that $m_{2}<m_{1}$ are two given nonnegative integers
with $m_{1}-m_{2}$ even. Let $n=m_{1}+m_{2}+1$ and let $I=(a,b),0<a<b$ be a
given open interval. Suppose that $f\in C^{n+2}(I)$ with $f^{(n+1)}(x)\neq 0$
for $x\in I$, and let $P$ and $Q$ satisfy the Hermite interpolation
conditions given by (\ref{PQ}). Then there is a unique $x_{0},a<x_{0}<b,$
such that $P(x_{0})=Q(x_{0})$.
\end{theorem}

\begin{proof}
We may assume, without loss of generality, that $f^{(n+1)}(x)>0$ on $I$. Let 
$E_{P}(x)=f(x)-P(x)$ and $E_{Q}(x)=f(x)-Q(x)$ denote the respective error
functions for $P$ and $Q$, and let $f[x_{0},x_{1},...,x_{n}]$ denote the $n$%
th order divided difference of $f$ for distinct nodes $x_{0},x_{1},...,x_{n}$%
. In general, divided differences at distinct points are defined inductively
by $f[x_{0},x_{1},...,x_{j}]=\tfrac{%
f[x_{0},x_{1},...,x_{j-1}]-f[x_{1},...,x_{j}]}{x_{0}-x_{j}}$ with $%
f[x_{0}]=f(x_{0})$. For sufficiently differentiable $f$, one can extend the
definition of divided difference in a continuous fashion when the nodes are
not all distinct (see, for example, \cite{ik}). We let $%
f[x_{0}^{m_{0}},x_{1}^{m_{1}},...,x_{n}^{m_{n}}]$ denote the divided
difference where $x_{k}$ appears $m_{k}$ times. Using one well--known form
of the error in Hermite interpolation, one has 
\begin{eqnarray}
E_{P}(x) &=&(x-a)^{m_{1}+1}(x-b)^{m_{2}+1}f[x,a^{m_{1}+1},b^{m_{2}+1}]\text{
and}  \label{epq} \\
E_{Q}(x) &=&(x-a)^{m_{2}+1}(x-b)^{m_{1}+1}f[x,a^{m_{2}+1},b^{m_{1}+1}] 
\notag
\end{eqnarray}%
Let 
\begin{equation*}
h_{1}(x)=f[x,a^{m_{1}+1},b^{m_{2}+1}],h_{2}(x)=f[x,a^{m_{2}+1},b^{m_{1}+1}].
\end{equation*}%
Now $P(x)=Q(x)\iff E_{P}(x)=E_{Q}(x)\iff $%
\begin{equation}
(x-a)^{m_{1}-m_{2}}h_{1}(x)=(x-b)^{m_{1}-m_{2}}h_{2}(x).  \label{h1h2}
\end{equation}%
By the Mean Value Theorem for divided differences (see \cite{ik}), $%
f[x,a^{m_{1}+1},b^{m_{2}+1}]=\tfrac{f^{(n+1)}(\zeta _{1})}{(n+1)!}$ and $%
f[x,a^{m_{2}+1},b^{m_{1}+1}]=\tfrac{f^{(n+1)}(\zeta _{2})}{(n+1)!}$, where $%
\zeta _{1},\zeta _{2}\in I$ if $x\in I$. Thus $%
f[x,a^{m_{1}+1},b^{m_{2}+1}]>0 $ and $f[x,a^{m_{2}+1},b^{m_{1}+1}]>0$. Now $%
h_{1}^{\prime }(x)=\tfrac{d}{dx}%
f[x,a^{m_{1}+1},b^{m_{2}+1}]=f[x,x,a^{m_{1}+1},b^{m_{2}+1}]$ (see \cite{ik}%
), which implies that $\tfrac{d}{dx}\left[ (x-a)^{m_{1}-m_{2}}h_{1}(x)\right]
=(x-a)^{m_{1}-m_{2}}h_{1}^{\prime }(x)+(m_{1}-m_{2})(x-a)^{m-1}h_{1}(x)=$

$(x-a)^{m_{1}-m_{2}-1}\left[
(x-a)f[x,x,a^{m_{1}+1},b^{m_{2}+1}]+(m_{1}-m_{2})f[x,a^{m_{1}+1},b^{m_{2}+1}]%
\right] $. Now $(x-a)^{m_{1}-m_{2}-1}\geq 0$ for $x\in I.$ Simplifying the
term in brackets using properties of divided differences yields $%
(x-a)f[x,x,b^{m_{2}+1},a^{m_{1}+1}]+(m_{1}-m_{2})f[x,b^{m_{2}+1},a^{m_{1}+1}]=f[x,x,b^{m_{2}+1},a^{m_{1}}]-f[x,b^{m_{2}+1},a^{m_{1}+1}]+ 
$

$%
(m_{1}-m_{2})f[x,b^{m_{2}+1},a^{m_{1}+1}]=f[x,x,b^{m_{2}+1},a^{m_{1}}]+(m_{1}-m_{2}-1)f[x,b^{m_{2}+1},a^{m_{1}+1}]>0 
$ again by the Mean Value Theorem for divided differences. Thus

$\dfrac{d}{dx}\left[ (x-a)^{m_{1}-m_{2}}h_{1}(x)\right] >0\Rightarrow
(x-a)^{m_{1}-m_{2}}h_{1}(x)$ is increasing on $I$. Similarly, $\dfrac{d}{dx}%
\left[ (x-b)^{m_{1}-m_{2}}h_{2}(x)\right] =(x-b)^{m_{1}-m_{2}}h_{2}^{\prime
}(x)+(m_{1}-m_{2})(x-b)^{m_{1}-m_{2}-1}h_{2}(x)=$

$(x-b)^{m_{1}-m_{2}-1}\left[
(x-b)f[x,x,a^{m_{2}+1},b^{m_{1}+1}]+(m_{1}-m_{2})f[x,a^{m_{2}+1},b^{m_{1}+1}]%
\right] =$

$(x-b)^{m_{1}-m_{2}-1}\left(
f[x,x,a^{m_{2}+1},b^{m_{1}}]+(m_{1}-m_{2}-1)f[x,a^{m_{2}+1},b^{m_{1}+1}]%
\right) $. Since $m_{1}-m_{2}-1$ is odd, $\dfrac{d}{dx}\left[
(x-b)^{m_{1}-m_{2}}h_{2}(x)\right] <0$ on $I$, which implies that $%
(x-b)^{m_{1}-m_{2}}h_{2}(x)$ is decreasing on $I$. Thus $%
(x-a)^{m_{1}-m_{2}}h_{1}(x)$ is positive and increasing on $I$ and vanishes
at $a$, while $(x-b)^{m_{1}-m_{2}}h_{2}(x)$ is positive and decreasing on $I$
and vanishes at $b$. Hence the equation in (\ref{h1h2}) has a unique
solution $x_{0}\in I$. Since $E_{P}(x_{0})=E_{Q}(x_{0}),P(x_{0})=Q(x_{0})$,
which finishes the proof of Theorem \ref{T1}.
\end{proof}

\begin{remark}
(1) Theorem 1 was proven in \cite{h1} using a different approach and only
for the case when $m_{1}-m_{2}=2$.

(2) Heuristically speaking, we may consider the means defined in \cite{h1}
as a special case of the means above, where $m_{1}=r$ and $m_{2}=-1$. The
latter value means that no values of $f$ or any of its derivatives are
matched. However, the formulas we use do not actually work if $m_{2}=-1$.
\end{remark}

The proof of the following theorem is almost identical to the proof of
Theorem \ref{T1} and we omit it.

\begin{theorem}
\label{T2}Suppose that $m_{2}<m_{1}$ are two given nonnegative integers with 
$m_{1}-m_{2}$ odd. Let $n=m_{1}+m_{2}+1$ and let $I=(a,b),0<a<b$ be a given
open interval. Suppose that $f\in C^{n+2}(I)$ with $f^{(n+1)}(x)\neq 0$ for $%
x\in I$, and let $P$ and $Q$ satisfy the Hermite interpolation conditions
given by (\ref{PQ}). Then there is a unique $x_{0},a<x_{0}<b,$ such that $%
f(x_{0})=\tfrac{1}{2}\left( P(x_{0})+Q(x_{0})\right) $.
\end{theorem}

\qquad The unique $x_{0}$ from Theorems \ref{T1} and \ref{T2} defines a
strict, symmetric mean, which we denote by $x_{0}=M_{f,m_{1},m_{2}}(a,b)$.
It is easy to unify the cases of $m_{1}-m_{2}$ even or odd as follows: $%
M_{f,m_{1},m_{2}}(a,b)$ is the unique solution, in $(a,b)$, of the equation $%
E_{P}(x)=(-1)^{m_{1}-m_{2}}E_{Q}(x)$. Equivalently, $M_{f,m_{1},m_{2}}(a,b)$
is the unique solution, in $(a,b)$, of the equation 
\begin{equation}
(x-a)^{m_{1}-m_{2}}f[x,a^{m_{1}+1},b^{m_{2}+1}]=(b-x)^{m_{1}-m_{2}}f[x,a^{m_{2}+1},b^{m_{1}+1}].
\label{meandef}
\end{equation}

As in \cite{h1} and \cite{h2}, we shall see that some of the familiar means,
such as the arithmetic, geometric, and harmonic means arise in certain
special cases. For $f(x)=x^{p},$ we denote $M_{f,m_{1},m_{2}}(a,b)$ by $%
M_{p,m_{1},m_{2}}(a,b)$ for any real number $p$ with $p\notin \left\{
0,1,...,n\right\} $. If $p=k,k\in \left\{ 0,1,...,n\right\} $, one can
define $M_{p,m_{1},m_{2}}$ using a limiting argument, or by defining $%
M_{p,m_{1},m_{2}}$ to be $M_{f,m_{1},m_{2}},$ where $f(x)=x^{k}\log x$. This
gives a continuous extension of $M_{p,m_{1},m_{2}}$ to all real numbers $p$.

\begin{remark}
For any polynomial $R\in \pi _{n},$ $%
M_{f-R,m_{1},m_{2}}(a,b)=M_{f,m_{1},m_{2}}(a,b)$.
\end{remark}

\qquad The following three theorems are the analogs of (\cite{h2}, Theorems
1.3 and 1.4) and (\cite{h1}, Theorem 1.8) for\ Hermite interpolation.

\begin{theorem}
\label{T3}If $p=m_{1}+m_{2}+2$, then $M_{p,m_{1},m_{2}}(a,b)=A(a,b)=\tfrac{%
a+b}{2}$.
\end{theorem}

\begin{proof}
If $f(x)=x^{m_{1}+m_{2}+2}$, then by the Mean Value Theorem for divided
differences, $%
f[x,a^{m_{1}+1},b^{m_{2}+1}]=f[x,a^{m_{2}+1},b^{m_{1}+1}]=(m_{1}+m_{2}+2)!$.
Thus the unique solution, $x_{0}$, in $I=(a,b)$ of the equation $%
E_{P}(x)=(-1)^{m_{1}-m_{2}}E_{Q}(x)$ is the unique solution of $%
(x-a)^{m_{1}-m_{2}}=(-1)^{m_{1}-m_{2}}(x-b)^{m_{1}-m_{2}}$, which implies
that $x_{0}=\tfrac{a+b}{2}$.
\end{proof}

\begin{theorem}
\qquad \label{T4}If $p=-1$, then $M_{p,m_{1},m_{2}}(a,b)=H(a,b)=\tfrac{2ab}{%
a+b}$ for any $m_{1}$ and $m_{2}$.
\end{theorem}

\begin{proof}
If $f(x)=\tfrac{1}{x}$, then $f[x_{0},x_{1},...,x_{n}]=\tfrac{(-1)^{n}}{%
x_{0}x_{1}\cdots x_{n}}$ (see \cite{ost}, page 11, formula (4)]. It then
follows easily that $f[x,a^{m_{1}+1},b^{m_{2}+1}]=\tfrac{(-1)^{m_{1}+m_{2}}}{%
a^{m_{1}+1}b^{m_{2}+1}x}$ and $f[x,a^{m_{2}+1},b^{m_{1}+1}]=\tfrac{%
(-1)^{m_{1}+m_{2}}}{a^{m_{2}+1}b^{m_{1}+1}x}$; Thus the unique solution, $%
x_{0}$, in $I=(a,b)$, of the equation $E_{P}(x)=(-1)^{m_{1}-m_{2}}E_{Q}(x)$
is the unique solution of $(x-a)^{m_{1}-m_{2}}\tfrac{(-1)^{m_{1}+m_{2}}}{%
a^{m_{1}+1}b^{m_{2}+1}x}=(-1)^{m_{1}-m_{2}}(x-b)^{m_{1}-m_{2}}\tfrac{%
(-1)^{m_{1}+m_{2}}}{a^{m_{2}+1}b^{m_{1}+1}x}$, which is equivalent to $%
(x-a)^{m_{1}-m_{2}}b^{m_{1}-m_{2}}=(x-b)^{m_{1}-m_{2}}a^{m_{1}-m_{2}}%
\Rightarrow x=\tfrac{2ab}{a+b}$.
\end{proof}

$(-1)^{m_{1}-m_{2}}E_{Q}(x)$ is the unique solution of 

\qquad Theorems \ref{T3} and \ref{T4} show that the arithmetic and harmonic
means arise as the $x$ coordinates of the intersection point of Hermite
interpolants. Our next result shows that the geometric mean arises as well,
but the proof is considerably more difficult.

\begin{theorem}
\label{T5}If $p=\tfrac{m_{1}+m_{2}+1}{2}$, where $m_{1}+m_{2}$ is even, then 
$M_{p,m_{1},m_{2}}(a,b)=G(a,b)=\sqrt{ab}$.
\end{theorem}

\begin{remark}
Theorem \ref{T5} does not hold if $m_{1}+m_{2}$ is odd. In that case, $p$ is
a positive integer strictly less than $m_{1}+m_{2}+1$, which implies that $%
f^{(m_{1}+m_{2}+2)}(x)\equiv 0$.
\end{remark}

Before proving Theorem \ref{T5}, we need three lemmas.

\begin{lemma}
\label{L0}Let $m_{1}\geq 0$ be any integer. Then 
\begin{equation}
\tsum\limits_{k=0}^{m_{1}}(-1)^{k}\tbinom{m_{1}/2}{k}(1-b)^{k}=b^{m_{1}}%
\tsum\limits_{k=0}^{m_{1}}\tbinom{m_{1}/2}{k}(1-b)^{k}b^{-k},b\neq 0
\label{0}
\end{equation}
\end{lemma}

\begin{proof}
We use induction in $m_{1}$. First, it is trivial that (\ref{0}) holds when $%
m_{1}=0$ or $m_{1}=1$. Now let $G_{m_{1}}(b)=\tsum%
\limits_{k=0}^{m_{1}}(-1)^{k}\tbinom{m_{1}/2}{k}(1-b)^{k}$ and $%
H_{m_{1}}(b)=b^{m_{1}}\tsum\limits_{k=0}^{m_{1}}\tbinom{m_{1}/2}{k}%
(1-b)^{k}b^{-k}$ denote the left and right hand sides of (\ref{0}),
respectively. Then $G_{m_{1}+2}(b)=\tsum\limits_{k=0}^{m_{1}+2}(-1)^{k}%
\tbinom{m_{1}/2+1}{k}(1-b)^{k}=(-1)^{m_{1}+1}\tbinom{m_{1}/2+1}{m_{1}+1}%
(1-b)^{m_{1}+1}+$

$(-1)^{m_{1}+2}\tbinom{m_{1}/2+1}{m_{1}+2}(1-b)^{m_{1}+2}+\tsum%
\limits_{k=0}^{m_{1}}(-1)^{k}\tbinom{m_{1}/2}{k}(1-b)^{k}+$

$\tsum\limits_{k=0}^{m_{1}-1}(-1)^{k+1}\tbinom{m_{1}/2}{k}%
(1-b)^{k+1}=(-1)^{m_{1}+1}\tbinom{m_{1}/2+1}{m_{1}+1}(1-b)^{m_{1}+1}+$

$(-1)^{m_{1}+2}\tbinom{m_{1}/2+1}{m_{1}+2}%
(1-b)^{m_{1}+2}+G_{m_{1}}(b)-(1-b)G_{m_{1}}(b)-$

$(-1)^{m_{1}+2}\tbinom{m_{1}/2}{m_{1}}%
(1-b)^{m_{1}+1}=bG_{m_{1}}(b)+(-1)^{m_{1}+1}\tbinom{m_{1}/2+1}{m_{1}+1}%
(1-b)^{m_{1}+1}+$

$(-1)^{m_{1}+2}\tbinom{m_{1}/2+1}{m_{1}+2}(1-b)^{m_{1}+2}-(-1)^{m_{1}+1}%
\tbinom{m_{1}/2}{m_{1}}(1-b)^{m_{1}+1}$

It is easy to show that $b^{m_{1}}G_{m_{1}}\left( \tfrac{1}{b}\right)
=H_{m_{1}}(b)$. Thus $H_{m_{1}+2}(b)=$

$b^{m_{1}+2}G_{m_{1}+2}\left( \tfrac{1}{b}\right) =b^{m_{1}+2}(-1)^{m_{1}+1}%
\tbinom{m_{1}/2+1}{m_{1}+1}(1-1/b)^{m_{1}+1}+$

$b^{m_{1}+2}(-1)^{m_{1}+2}\tbinom{m_{1}/2+1}{m_{1}+2}%
(1-1/b)^{m_{1}+2}+b^{m_{1}+1}G_{m_{1}}\left( \tfrac{1}{b}\right) -$

$b^{m_{1}+2}(-1)^{m_{1}+1}\tbinom{m_{1}/2}{m_{1}}%
(1-1/b)^{m_{1}+1}=bH_{m_{1}}(b)+$

$b(-1)^{m_{1}+1}\tbinom{m_{1}/2+1}{m_{1}+1}(b-1)^{m_{1}+1}+$

$(-1)^{m_{1}+2}\tbinom{m_{1}/2+1}{m_{1}+2}(b-1)^{m_{1}+2}-b(-1)^{m_{1}+1}%
\tbinom{m_{1}/2}{m_{1}}(b-1)^{m_{1}+1}$. Assuming that $%
G_{m_{1}}(b)=H_{m_{1}}(b)$, we have that $G_{m_{1}+2}(b)=H_{m_{1}+2}(b)\iff
(-1)^{m_{1}+1}\tbinom{m_{1}/2+1}{m_{1}+1}(1-b)^{m_{1}+1}+$

$(-1)^{m_{1}+2}\tbinom{m_{1}/2+1}{m_{1}+2}(1-b)^{m_{1}+2}-(-1)^{m_{1}+1}%
\tbinom{m_{1}/2}{m_{1}}(1-b)^{m_{1}+1}=$

$b(-1)^{m_{1}+1}\tbinom{m_{1}/2+1}{m_{1}+1}(b-1)^{m_{1}+1}+$

$(-1)^{m_{1}+2}\tbinom{m_{1}/2+1}{m_{1}+2}(b-1)^{m_{1}+2}-b(-1)^{m_{1}+1}%
\tbinom{m_{1}/2}{m_{1}}(b-1)^{m_{1}+1}$. If $m_{1}$ is even, then the
equality holds trivially since all terms involved are $0$. So assume now
that $m_{1}$ is odd. Then

$G_{m_{1}+2}(b)=H_{m_{1}+2}(b)\iff \tbinom{m_{1}/2+1}{m_{1}+1}%
(1-b)^{m_{1}+1}-\tbinom{m_{1}/2+1}{m_{1}+2}(1-b)^{m_{1}+2}-\tbinom{m_{1}/2}{%
m_{1}}(1-b)^{m_{1}+1}=b\tbinom{m_{1}/2+1}{m_{1}+1}(b-1)^{m_{1}+1}-\tbinom{%
m_{1}/2+1}{m_{1}+2}(b-1)^{m_{1}+2}-b\tbinom{m_{1}/2}{m_{1}}%
(b-1)^{m_{1}+1}\iff $

$\tbinom{m_{1}/2+1}{m_{1}+1}(1-b)^{m_{1}+2}-2\tbinom{m_{1}/2+1}{m_{1}+2}%
(1-b)^{m_{1}+2}-\tbinom{m_{1}/2}{m_{1}}(1-b)^{m_{1}+2}=\allowbreak 0\iff $

$\tbinom{m_{1}/2+1}{m_{1}+1}-2\tbinom{m_{1}/2+1}{m_{1}+2}-\tbinom{m_{1}/2}{%
m_{1}}=\allowbreak 0\iff \tbinom{m_{1}/2}{m_{1}+1}-2\tbinom{m_{1}/2+1}{%
m_{1}+2}=\allowbreak 0\iff $

$\tfrac{\left( \tfrac{m_{1}}{2}\right) \left( \tfrac{m_{1}}{2}-1\right)
\cdots \left( \tfrac{m_{1}}{2}-m_{1}\right) }{(m_{1}+1)!}-2\tfrac{\left( 
\tfrac{m_{1}}{2}+1\right) \left( \tfrac{m_{1}}{2}-1\right) \cdots \left( 
\tfrac{m_{1}}{2}-m_{1}\right) }{(m_{1}+2)!}=0\iff $

$\tfrac{m_{1}(m_{1}-2)\cdots (m_{1}-2m_{1})}{2^{m_{1}+1}(m_{1}+1)!}-\tfrac{%
(m_{1}+2)(m_{1}-2)\cdots (m_{1}-2m_{1})}{2^{m_{1}+1}(m_{1}+2)!}=0\iff $

$(m_{1}+2)m_{1}(m_{1}-2)\cdots (-m_{1})-(m_{1}+2)m_{1}(m_{1}-2)\cdots
(-m_{1})=0$. That completes the proof of Lemma \ref{L0}.
\end{proof}

\begin{lemma}
\label{L1}Let $m_{1}$ and $m_{2}$ be any integers. Then if $y>0,$%
\begin{equation*}
\tsum\limits_{k=0}^{m_{1}}\tbinom{m_{2}+m_{1}-k}{m_{2}}\tbinom{%
(m_{2}+m_{1})/2}{k}(-1)^{k}(1-y)^{k}=y^{m_{1}}\tsum\limits_{k=0}^{m_{1}}%
\tbinom{m_{2}+m_{1}-k}{m_{2}}\tbinom{(m_{2}+m_{1})/2}{k}(1-y)^{k}y^{-k}
\end{equation*}
\end{lemma}

\begin{proof}
It is not hard to show, using, for example, the methods in \cite{pwz}, that $%
\tfrac{1}{\tbinom{m_{1}+m_{2}}{m_{2}}}\sum\limits_{k=0}^{m_{1}}\binom{%
m_{2}+m_{1}-k}{m_{2}}\binom{(m_{2}+m_{1})/2}{k}x^{k}=\,_{2}F_{1}\left( -%
\tfrac{1}{2}m_{1}-\tfrac{1}{2}m_{2},-m_{1}\text{;}-m_{1}-m_{2}\text{;}%
-x\right) $, where $_{2}F_{1}\left( [a,b],[c],z\right) $ is the
hypergeometric function $\sum\limits_{k\geq 0}\tfrac{(a)_{k}(b)_{k}}{(c)_{k}}%
\tfrac{z^{k}}{k!}$. Thus it suffices to prove that $_{2}F_{1}\left( -\tfrac{1%
}{2}m_{1}-\tfrac{1}{2}m_{2},-m_{1}\text{;}-m_{1}-m_{2}\text{;}1-y\right) =$

$y^{m_{1}}\,_{2}F_{1}\left( -\tfrac{1}{2}m_{1}-\tfrac{1}{2}m_{2},-m_{1}\text{%
;}-m_{1}-m_{2}\text{; }\tfrac{y-1}{y}\right) $. The latter equality follows
from the identity $(1-x)^{b}\,_{2}F_{1}\left( a,b\text{; }c\text{; }x\right)
=\,_{2}F_{1}\left( c-a,b\text{; }c\text{; }\tfrac{x}{x-1}\right) ,x\notin
(1,\infty )$ with $a=-\tfrac{1}{2}m_{1}-\tfrac{1}{2}%
m_{2},b=-m_{1},c=-m_{1}-m_{2}$, and $x=1-y$. That proves Lemma \ref{L1}.
\end{proof}

We now use Lemmas \ref{L0} and \ref{L1} to prove the following identity.

\begin{lemma}
\label{L2}If $m_{1}\geq 0$ and $m_{2}$ are any integers, if $p=\tfrac{%
m_{1}+m_{2}+1}{2}$, and if $b\neq -1$ and $b\neq 0$, then 
\begin{equation*}
\tsum\limits_{k=0}^{m_{1}}\tsum\limits_{l=0}^{m_{1}-k}\tbinom{m_{2}+l}{m_{2}}%
(-1)^{k}\tbinom{p}{k}(1-b)^{k}(1+b)^{-l}=b^{m_{1}}\tsum\limits_{k=0}^{m_{1}}%
\tsum\limits_{l=0}^{m_{1}-k}\tbinom{m_{2}+l}{m_{2}}\tbinom{p}{k}%
(1-b)^{k}(1+b)^{-l}b^{l-k}.
\end{equation*}
\end{lemma}

\begin{remark}
The lemma actually holds if $m_{1}$ is a negative integer if one interprets
both sides of the equality to be $0$ in that case.
\end{remark}

\begin{proof}
We denote the left and right hand sides in Lemma \ref{L2} by%
\begin{eqnarray*}
L_{m_{1},m_{2}}(b) &=&\tsum\limits_{k=0}^{m_{1}}\tsum\limits_{l=0}^{m_{1}-k}%
\tbinom{m_{2}+l}{l}\tbinom{p}{k}(-1)^{k}(1-b)^{k}(1+b)^{-l}, \\
R_{m_{1},m_{2}}(b)
&=&b^{m_{1}}\tsum\limits_{k=0}^{m_{1}}\tsum\limits_{l=0}^{m_{1}-k}\tbinom{%
m_{2}+l}{l}\tbinom{p}{k}(1-b)^{k}(1+b)^{-l}b^{l-k},
\end{eqnarray*}%
respectively. Here we used the identity $\tbinom{m_{2}+l}{m_{2}}=\tbinom{%
m_{2}+l}{l}$. Thus, to prove Lemma \ref{L2}, it suffices to prove 
\begin{equation}
L_{m_{1},m_{2}}(b)=R_{m_{1},m_{2}}(b).  \label{id}
\end{equation}

We shall first prove the recursion 
\begin{gather}
L_{m_{1}+1,m_{2}-1}(b)=  \label{lid} \\
\tfrac{b}{b+1}L_{m_{1},m_{2}}(b)+\tsum\limits_{k=0}^{m_{1}+1}\tbinom{%
m_{2}+m_{1}+1-k}{m_{2}}\tbinom{p}{k}(-1)^{k}(1-b)^{k}(1+b)^{-m_{1}-1+k} 
\notag
\end{gather}%
Throughout we use the identities $\binom{n+1}{k}=\binom{n}{k}+\binom{n}{k-1}$%
, which, in particular, implies that $\binom{m_{2}+l-1}{l}=\binom{m_{2}+l}{l}%
-\binom{m_{2}+l-1}{l-1}$. In addition we use the fact that $\binom{n}{-1}=0$
for any whole number, $n$. Now $L_{m_{1}+1,m_{2}-1}(b)=\sum%
\limits_{k=0}^{m_{1}+1}\sum\limits_{l=0}^{m_{1}+1-k}\binom{m_{2}+l-1}{l}%
\tbinom{p}{k}(-1)^{k}(1-b)^{k}(1+b)^{-l}=$

$\binom{p}{m_{1}+1}(-1)^{m_{1}+1}(1-b)^{m_{1}+1}+\sum\limits_{k=0}^{m_{1}}%
\sum\limits_{l=0}^{m_{1}+1-k}\binom{m_{2}+l-1}{l}\tbinom{p}{k}%
(-1)^{k}(1-b)^{k}(1+b)^{-l}=$

$\binom{p}{m_{1}+1}(-1)^{m_{1}+1}(1-b)^{m_{1}+1}+\sum\limits_{k=0}^{m_{1}}%
\sum\limits_{l=0}^{m_{1}+1-k}\binom{m_{2}+l}{l}\tbinom{p}{k}%
(-1)^{k}(1-b)^{k}(1+b)^{-l}-\sum\limits_{k=0}^{m_{1}}\sum%
\limits_{l=0}^{m_{1}+1-k}\binom{m_{2}+l-1}{l-1}\tbinom{p}{k}%
(-1)^{k}(1-b)^{k}(1+b)^{-l}=$

$\binom{p}{m_{1}+1}(-1)^{m_{1}+1}(1-b)^{m_{1}+1}+\sum\limits_{k=0}^{m_{1}}%
\binom{m_{2}+m_{1}+1-k}{m_{2}}\binom{p}{k}%
(-1)^{k}(1-b)^{k}(1+b)^{-m_{1}-1+k}+$

$\sum\limits_{k=0}^{m_{1}}\sum\limits_{l=0}^{m_{1}-k}\binom{m_{2}+l}{l}%
\tbinom{p}{k}(-1)^{k}(1-b)^{k}(1+b)^{-l}-\sum\limits_{k=0}^{m_{1}}\binom{%
m_{2}+m_{1}-k}{m_{2}}\binom{p}{k}(-1)^{k}(1-b)^{k}(1+b)^{-m_{1}-1+k}-$

$\sum\limits_{k=0}^{m_{1}}\sum\limits_{l=0}^{m_{1}-k}\binom{m_{2}+l-1}{l-1}%
\tbinom{p}{k}(-1)^{k}(1-b)^{k}(1+b)^{-l}=\sum\limits_{k=0}^{m_{1}+1}\binom{%
m_{2}+m_{1}-k}{m_{2}-1}\binom{p}{k}(-1)^{k}(1-b)^{k}(1+b)^{-m_{1}-1+k}+$

$\sum\limits_{k=0}^{m_{1}}\sum\limits_{l=0}^{m_{1}-k}\binom{m_{2}+l}{l}%
\tbinom{p}{k}(-1)^{k}(1-b)^{k}(1+b)^{-l}-\sum\limits_{k=0}^{m_{1}}\sum%
\limits_{l=-1}^{m_{1}-k}\binom{m_{2}+l}{l}\tbinom{p}{k}%
(-1)^{k}(1-b)^{k}(1+b)^{-l-1}=$

$\sum\limits_{k=0}^{m_{1}+1}\binom{m_{2}+m_{1}-k}{m_{2}-1}\binom{p}{k}%
(-1)^{k}(1-b)^{k}(1+b)^{-m_{1}-1+k}+\sum\limits_{k=0}^{m_{1}}\sum%
\limits_{l=0}^{m_{1}-k}\binom{m_{2}+l}{l}\tbinom{p}{k}%
(-1)^{k}(1-b)^{k}(1+b)^{-l}-$

$(1+b)^{-1}\sum\limits_{k=0}^{m_{1}}\sum\limits_{l=0}^{m_{1}-k}\binom{m_{2}+l%
}{l}\tbinom{p}{k}(-1)^{k}(1-b)^{k}(1+b)^{-l}=\sum\limits_{k=0}^{m_{1}+1}%
\binom{m_{2}+m_{1}-k}{m_{2}-1}\binom{p}{k}%
(-1)^{k}(1-b)^{k}(1+b)^{-m_{1}-1+k}+L_{m_{1},m_{2}}(b)-(1+b)^{-1}L_{m_{1},m_{2}}(b)= 
$

$\tfrac{b}{b+1}L_{m_{1},m_{2}}(b)+\sum\limits_{k=0}^{m_{1}+1}\binom{%
m_{2}+m_{1}+1-k}{m_{2}}\binom{p}{k}(-1)^{k}(1-b)^{k}(1+b)^{-m_{1}-1+k}$.
That proves (\ref{lid}). It is easy to show that 
\begin{equation}
b^{m_{1}}L_{m_{1},m_{2}}\left( \tfrac{1}{b}\right) =R_{m_{1},m_{2}}(b).
\label{lrid}
\end{equation}%
Thus by (\ref{lrid}) and (\ref{lid}), $%
R_{m_{1}+1,m_{2}-1}(b)=b^{m_{1}+1}L_{m_{1}+1,m_{2}-1}\left( \tfrac{1}{b}%
\right) =$

$b^{m_{1}+1}\left( \tfrac{1}{b+1}L_{m_{1},m_{2}}\left( \tfrac{1}{b}\right)
+\sum\limits_{k=0}^{m_{1}+1}\binom{m_{2}+m_{1}+1-k}{m_{2}}\binom{p}{k}%
(-1)^{k}(1-1/b)^{k}(1+1/b)^{-m_{1}-1+k}\right) =$

$b^{m_{1}+1}\left( \tfrac{1}{b+1}L_{m_{1},m_{2}}\left( \tfrac{1}{b}\right)
+\sum\limits_{k=0}^{m_{1}+1}\binom{m_{2}+m_{1}+1-k}{m_{2}}\binom{p}{k}%
(1-b)^{k}(1+b)^{-m_{1}-1+k}b^{m_{1}+1-2k}\right) =$

$\left( \tfrac{b^{m_{1}+1}}{b+1}L_{m_{1},m_{2}}\left( \tfrac{1}{b}\right)
+\sum\limits_{k=0}^{m_{1}+1}\binom{m_{2}+m_{1}+1-k}{m_{2}}\binom{p}{k}%
(1-b)^{k}(1+b)^{-m_{1}-1+k}b^{2m_{1}+2-2k}\right) =$

$\tfrac{b}{b+1}R_{m_{1},m_{2}}(b)+\sum\limits_{k=0}^{m_{1}+1}\binom{%
m_{2}+m_{1}+1-k}{m_{2}}\binom{p}{k}%
(1-b)^{k}(1+b)^{-m_{1}-1+k}b^{2m_{1}+2-2k} $. That yields the recursion 
\begin{gather}
R_{m_{1}+1,m_{2}-1}(b)=  \label{rid} \\
\tfrac{b}{b+1}R_{m_{1},m_{2}}(b)+\tsum\limits_{k=0}^{m_{1}+1}\tbinom{%
m_{2}+m_{1}+1-k}{m_{2}}\tbinom{p}{k}%
(1-b)^{k}(1+b)^{-m_{1}-1+k}b^{2m_{1}+2-2k}.  \notag
\end{gather}

In a similar fashion, one can also prove the recursions%
\begin{gather}
L_{m_{1}-1,m_{2}+1}(b)=\tfrac{b+1}{b}L_{m_{1},m_{2}}(b)-\tfrac{b+1}{b}%
(-1)^{m_{1}}\tbinom{p}{m_{1}}(1-b)^{m_{1}}  \label{lid2} \\
-\tfrac{b+1}{b}\tsum\limits_{k=0}^{m_{1}-1}(-1)^{k}\tbinom{m_{1}+m_{2}-k+1}{%
m_{2}+1}\tbinom{p}{k}(1-b)^{k}(1+b)^{-(m_{1}-k)}  \notag
\end{gather}

and%
\begin{gather}
R_{m_{1}-1,m_{2}+1}(b)=\tfrac{b+1}{b}R_{m_{1},m_{2}}(b)-\tfrac{b+1}{b}%
(-1)^{m_{1}}\tbinom{p}{m_{1}}(b-1)^{m_{1}}  \label{rid2} \\
-\tfrac{b+1}{b}\tsum\limits_{k=0}^{m_{1}-1}(-1)^{k}\tbinom{m_{1}+m_{2}-k+1}{%
m_{2}+1}\tbinom{p}{k}(b-1)^{k}(1+b)^{-m_{1}+k}b^{2m_{1}-2k}  \notag
\end{gather}%
We now use induction to prove (\ref{id}). We start the induction with $%
m_{2}=-1$ and $m_{1}$ any fixed non--negative integer. In that case the only
nonzero term on both sides of (\ref{id}) occurs when $l=0$, which yields $%
\tsum\limits_{k=0}^{m_{1}}(-1)^{k}\tbinom{m_{1}/2}{k}(1-b)^{k}=b^{m_{1}}%
\tsum\limits_{k=0}^{m_{1}}\tbinom{m_{1}/2}{k}(1-b)^{k}b^{-k}$, which is
precisely Lemma \ref{L0}. Proceeding with the induction, we assume now that $%
L_{m_{1},m_{2}}(b)=R_{m_{1},m_{2}}(b)$. Then, using (\ref{lid}) and (\ref%
{rid}), it follows that $L_{m_{1}+1,m_{2}-1}(b)=R_{m_{1}+1,m_{2}-1}(b)\iff
\sum\limits_{k=0}^{m_{1}+1}\binom{m_{2}+m_{1}+1-k}{m_{2}}\binom{p}{k}%
(-1)^{k}(1-b)^{k}(1+b)^{-m_{1}-1+k}=$

$\sum\limits_{k=0}^{m_{1}+1}\binom{m_{2}+m_{1}+1-k}{m_{2}}\binom{p}{k}%
(1-b)^{k}(1+b)^{-m_{1}-1+k}b^{2m_{1}+2-2k}\iff $%
\begin{gather}
\tsum\limits_{k=0}^{m_{1}+1}\tbinom{m_{2}+m_{1}+1-k}{m_{2}}\tbinom{p}{k}%
(-1)^{k}(1-b^{2})^{k}=  \label{1} \\
\left( b^{2}\right) ^{m_{1}+1}\tsum\limits_{k=0}^{m_{1}+1}\tbinom{%
m_{2}+m_{1}+1-k}{m_{2}}\tbinom{p}{k}(1-b^{2})^{k}\left( b^{2}\right) ^{-k} 
\notag
\end{gather}%
(\ref{1}) now follows from Lemma \ref{L1}, by replacing $m_{1}+1$ by $m_{1}$
and letting $y=b^{2}$. That proves, with the assumption $%
L_{m_{1},m_{2}}(b)=R_{m_{1},m_{2}}(b)$, that

\begin{equation}
L_{m_{1}+1,m_{2}-1}(b)=R_{m_{1}+1,m_{2}-1}(b).  \label{plus1}
\end{equation}

Also, using (\ref{lid2}) and (\ref{rid2}), it follows that $%
L_{m_{1}-1,m_{2}+1}(b)=R_{m_{1}-1,m_{2}+1}(b)\iff (-1)^{m_{1}}\tbinom{p}{%
m_{1}}(1-b)^{m_{1}}+\tsum\limits_{k=0}^{m_{1}-1}(-1)^{k}\tbinom{%
m_{1}+m_{2}-k+1}{m_{2}+1}\tbinom{p}{k}(1-b)^{k}(1+b)^{-(m_{1}-k)}=$

$(-1)^{m_{1}}\tbinom{p}{m_{1}}(b-1)^{m_{1}}+\tsum%
\limits_{k=0}^{m_{1}-1}(-1)^{k}\tbinom{m_{1}+m_{2}-k+1}{m_{2}+1}\tbinom{p}{k}%
(b-1)^{k}(1+b)^{-m_{1}+k}b^{2m_{1}-2k}$. To prove this equality we consider
two cases.

\textbf{Case 1:} $m_{1}$ is even

We must show that $\tsum\limits_{k=0}^{m_{1}-1}(-1)^{k}\tbinom{%
m_{1}+m_{2}-k+1}{m_{2}+1}\tbinom{p}{k}(1-b^{2})^{k}=$

$b^{2m_{1}}\tsum\limits_{k=0}^{m_{1}-1}\tbinom{m_{1}+m_{2}-k+1}{m_{2}+1}%
\tbinom{p}{k}\left( \tfrac{1-b^{2}}{b^{2}}\right) ^{k}$, which is equivalent
to

$\tsum\limits_{k=0}^{m_{1}-1}(-1)^{k}\tbinom{m_{1}+m_{2}-k}{m_{2}}\tbinom{%
(m_{2}+m_{1})/2}{k}(1-y)^{k}=y^{m_{1}}\tsum\limits_{k=0}^{m_{1}-1}\tbinom{%
m_{1}+m_{2}-k}{m_{2}}\tbinom{(m_{2}+m_{1})/2}{k}\left( \tfrac{1-y}{y}\right)
^{k}$ upon replacing $m_{2}$ by $m_{2}-1$ and letting $y=b^{2}$. Hence we
must prove that

$\tsum\limits_{k=0}^{m_{1}}(-1)^{k}\tbinom{m_{1}+m_{2}-k}{m_{2}}\tbinom{%
(m_{2}+m_{1})/2}{k}(1-y)^{k}-(-1)^{m_{1}}\tbinom{(m_{2}+m_{1})/2}{m_{1}}%
(1-y)^{m_{1}}=$

$y^{m_{1}}\tsum\limits_{k=0}^{m_{1}}\tbinom{m_{1}+m_{2}-k}{m_{2}}\tbinom{%
(m_{2}+m_{1})/2}{k}\left( \tfrac{1-y}{y}\right) ^{k}-y^{m_{1}}\tbinom{%
(m_{2}+m_{1})/2}{m_{1}}\left( \tfrac{1-y}{y}\right) ^{m_{1}}\iff $

$\tsum\limits_{k=0}^{m_{1}}(-1)^{k}\tbinom{m_{1}+m_{2}-k}{m_{2}}\tbinom{%
(m_{2}+m_{1})/2}{k}(1-y)^{k}=y^{m_{1}}\tsum\limits_{k=0}^{m_{1}}\tbinom{%
m_{1}+m_{2}-k}{m_{2}}\tbinom{(m_{2}+m_{1})/2}{k}\left( \tfrac{1-y}{y}\right)
^{k}$, which follows from Lemma \ref{L1}.

\textbf{Case 2:} $m_{1}$ is odd

We must show that $-\tbinom{p}{m_{1}}(1-b)^{m_{1}}+(1+b)^{-m_{1}}\tsum%
\limits_{k=0}^{m_{1}-1}(-1)^{k}\tbinom{m_{1}+m_{2}-k+1}{m_{2}+1}\tbinom{p}{k}%
(1-b^{2})^{k}=$

$\tbinom{p}{m_{1}}(1-b)^{m_{1}}+(1+b)^{-m_{1}}b^{2m_{1}}\tsum%
\limits_{k=0}^{m_{1}-1}\tbinom{m_{1}+m_{2}-k+1}{m_{2}+1}\tbinom{p}{k}\left( 
\tfrac{1-b^{2}}{b^{2}}\right) ^{k}\iff $

$-\tbinom{p}{m_{1}}(1-b^{2})^{m_{1}}+\tsum\limits_{k=0}^{m_{1}-1}(-1)^{k}%
\tbinom{m_{1}+m_{2}-k+1}{m_{2}+1}\tbinom{p}{k}(1-b^{2})^{k}=$

$\tbinom{p}{m_{1}}(1-b^{2})^{m_{1}}+b^{2m_{1}}\tsum\limits_{k=0}^{m_{1}-1}%
\tbinom{m_{1}+m_{2}-k+1}{m_{2}+1}\tbinom{p}{k}\left( \tfrac{1-b^{2}}{b^{2}}%
\right) ^{k}$. Replace $m_{2}$ by $m_{2}-1$ and let $y=b^{2}$ to get $%
\tbinom{(m_{2}+m_{1})/2}{m_{1}}(1-y)^{m_{1}}+\tsum%
\limits_{k=0}^{m_{1}-1}(-1)^{k}\tbinom{m_{1}+m_{2}-k}{m_{2}}\tbinom{%
(m_{2}+m_{1})/2}{k}(1-y)^{k}=$

$\tbinom{(m_{2}+m_{1})/2}{m_{1}}(1-y)^{m_{1}}+y^{m_{1}}\tsum%
\limits_{k=0}^{m_{1}-1}\tbinom{m_{1}+m_{2}-k}{m_{2}}\tbinom{(m_{2}+m_{1})/2}{%
k}\left( \tfrac{1-y}{y}\right) ^{k}\iff $

$-\tbinom{(m_{2}+m_{1})/2}{m_{1}}(1-y)^{m_{1}}+\tsum%
\limits_{k=0}^{m_{1}}(-1)^{k}\tbinom{m_{1}+m_{2}-k}{m_{2}}\tbinom{%
(m_{2}+m_{1})/2}{k}(1-y)^{k}+\tbinom{(m_{2}+m_{1})/2}{m_{1}}(1-y)^{m_{1}}=$

$\tbinom{(m_{2}+m_{1})/2}{m_{1}}(1-y)^{m_{1}}+y^{m_{1}}\tsum%
\limits_{k=0}^{m_{1}}\tbinom{m_{1}+m_{2}-k}{m_{2}}\tbinom{(m_{2}+m_{1})/2}{k}%
\left( \tfrac{1-y}{y}\right) ^{k}-y^{m_{1}}\tbinom{(m_{2}+m_{1})/2}{m_{1}}(%
\tfrac{1-y}{y})^{m_{1}}\iff \tsum\limits_{k=0}^{m_{1}}(-1)^{k}\tbinom{%
m_{1}+m_{2}-k}{m_{2}}\tbinom{(m_{2}+m_{1})/2}{k}(1-y)^{k}=y^{m_{1}}\tsum%
\limits_{k=0}^{m_{1}}\tbinom{m_{1}+m_{2}-k}{m_{2}}\tbinom{(m_{2}+m_{1})/2}{k}%
\left( \tfrac{1-y}{y}\right) ^{k}$, which is again Lemma \ref{L1}. That
proves, with the assumption $L_{m_{1},m_{2}}(b)=R_{m_{1},m_{2}}(b)$, that 
\begin{equation}
L_{m_{1}-1,m_{2}+1}(b)=R_{m_{1}-1,m_{2}+1}(b).  \label{minus1}
\end{equation}%
The case $m_{2}=-1$ and (\ref{minus1}) shows that (\ref{id}) holds when $%
m_{1}=m_{2}$ and $m_{1}$ is any non--negative integer, or when $%
m_{1}=m_{2}+1 $ and $m_{1}$ is any non--negative integer. (\ref{plus1}) now
shows that (\ref{id}) holds when $m_{1}\geq 0$ and $m_{2}$ are any integers.
That finishes the proof of Lemma \ref{L2}.
\end{proof}

We are now ready to prove Theorem \ref{T5}.

\begin{proof}
We now use a formula due to Spitzbart (see \cite{s}, Theorem 2), which
expresses divided differences of the form $%
f[x_{0}^{r_{0}+1},x_{1}^{r_{1}+1},...,x_{n}^{r_{n}+1}]$ with confluent
arguments as a linear combination of the values of $f$ and its derivatives
at $x_{0},x_{1},...,x_{n}$. Using $%
f(x)=x^{p},x_{0}=x,x_{1}=a,x_{2}=b,r_{0}=0,r_{1}=m_{1},$ and $r_{2}=m_{2}$,
one can write 
\begin{equation}
f[x,a^{m_{1}+1},b^{m_{2}+1}]=A_{1}+B_{1}+C_{1},  \label{2}
\end{equation}%
where $A_{1}=\sum\limits_{k=0}^{m_{1}}\sum\limits_{l=0}^{m_{1}-k}\binom{%
-m_{2}-1}{l}\binom{-1}{m_{1}-k-l}\binom{p}{k}%
(a-b)^{-m_{2}-1-l}(a-x)^{-m_{1}-1+k+l}a^{p-k},$ $B_{1}=\sum%
\limits_{k=0}^{m_{2}}\sum\limits_{l=0}^{m_{2}-k}\binom{-m_{1}-1}{l}\binom{-1%
}{m_{2}-k-l}\binom{p}{k}(b-a)^{-m_{1}-1-l}(b-x)^{-m_{2}-1+k+l}b^{p-k}$, and 
\begin{equation}
C_{1}=(x-a)^{-m_{1}-1}(x-b)^{-m_{2}-1}x^{p}.  \label{c1}
\end{equation}%
Using the identities $\binom{-1}{m_{j}-k-l}=(-1)^{m_{j}}(-1)^{k+l}$ and $%
\binom{-m_{j}-1}{l}=(-1)^{l}\binom{m_{j}+l}{m_{j}},j=1,2$ to simplify the
expressions for $A_{1}$ and $B_{1}$ yields

\begin{gather}
A_{1}=a^{p}\tsum\limits_{k=0}^{m_{1}}\tsum\limits_{l=0}^{m_{1}-k}\tbinom{%
m_{2}+l}{m_{2}}\tbinom{p}{k}%
(-1)^{m_{1}+k}(a-b)^{-m_{2}-1-l}(a-x)^{-m_{1}-1+k+l}a^{-k},  \label{a1b1} \\
B_{1}=b^{p}\tsum\limits_{k=0}^{m_{2}}\tsum\limits_{l=0}^{m_{2}-k}\tbinom{%
m_{1}+l}{m_{1}}\tbinom{p}{k}%
(-1)^{m_{2}+k}(b-a)^{-m_{1}-1-l}(b-x)^{-m_{2}-1+k+l}b^{-k}  \notag
\end{gather}

By switching $m_{1}$ and $m_{2}$ we obtain 
\begin{equation}
f[x,a^{m_{2}+1},b^{m_{1}+1}]=A_{2}+B_{2}+C_{2},  \label{3}
\end{equation}%
\ where 
\begin{eqnarray}
A_{2} &=&a^{p}\tsum\limits_{k=0}^{m_{2}}\tsum\limits_{l=0}^{m_{2}-k}\tbinom{%
m_{1}+l}{m_{1}}\tbinom{p}{k}%
(-1)^{m_{2}+k}(a-b)^{-m_{1}-1-l}(a-x)^{-m_{2}-1+k+l}a^{-k},  \notag \\
&&  \label{a2b2c2} \\
B_{2} &=&b^{p}\tsum\limits_{k=0}^{m_{1}}\tsum\limits_{l=0}^{m_{1}-k}\tbinom{%
m_{2}+l}{m_{2}}\tbinom{p}{k}%
(-1)^{m_{1}+k}(b-a)^{-m_{2}-1-l}(b-x)^{-m_{1}-1+k+l}b^{-k},  \notag \\
C_{2} &=&(x-a)^{-m_{2}-1}(x-b)^{-m_{1}-1}x^{p}.  \notag
\end{eqnarray}%
Now letting $x=1$ and $a=\tfrac{1}{b}$ in (\ref{2}) and (\ref{3}) yields%
\begin{eqnarray}
b^{m_{2}}\left( A_{1}+B_{1}+C_{1}\right)
&=&b^{m_{2}}f[1,(1/b)^{m_{1}+1},b^{m_{2}+1}]  \label{4} \\
b^{m_{1}}\left( A_{2}+B_{2}+C_{2}\right)
&=&b^{m_{1}}f[1,(1/b)^{m_{2}+1},b^{m_{1}+1}],  \notag
\end{eqnarray}%
After some simplification, we have $%
b^{m_{2}}A_{1}=(-1)^{m_{1}}b^{m_{1}+2m_{2}+2-p}(1-b)^{-m_{1}-m_{2}-2}(1+b)^{-m_{2}-1}\sum\limits_{k=0}^{m_{1}}\sum\limits_{l=0}^{m_{1}-k}%
\binom{m_{2}+l}{m_{2}}(-1)^{k}\binom{p}{k}(1-b)^{k}(1+b)^{-l},$

$%
b^{m_{2}}B_{1}=(-1)^{m_{1}}b^{m_{1}+m_{2}+p+1}(1-b)^{-m_{1}-m_{2}-2}(1+b)^{-m_{1}-1}\sum\limits_{k=0}^{m_{2}}\sum\limits_{l=0}^{m_{2}-k}%
\binom{m_{1}+l}{m_{1}}\binom{p}{k}(1-b)^{k}(1+b)^{-l}b^{l-k},$

$%
b^{m_{2}}C_{1}=(-1)^{m_{1}+1}(1-b)^{-m_{1}-m_{2}-2}b^{m_{1}+m_{2}+1},b^{m_{1}}A_{2}=(-1)^{m_{2}}b^{m_{2}+2m_{1}+2-p}(1-b)^{-m_{1}-m_{2}-2}(1+b)^{-m_{1}-1}\sum\limits_{k=0}^{m_{2}}\sum\limits_{l=0}^{m_{2}-k}%
\binom{m_{1}+l}{m_{1}}(-1)^{k}\binom{p}{k}%
(1-b)^{k}(1+b)^{-l},b^{m_{1}}B_{2}=(-1)^{m_{2}}b^{m_{1}+m_{2}+p+1}(1-b)^{-m_{1}-m_{2}-2}(1+b)^{-m_{2}-1}\sum\limits_{k=0}^{m_{1}}\sum\limits_{l=0}^{m_{1}-k}%
\binom{m_{2}+l}{m_{2}}\binom{p}{k}(1-b)^{k}(1+b)^{-l}b^{l-k},$ and $%
b^{m_{1}}C_{2}=(-1)^{m_{2}+1}(1-b)^{-m_{1}-m_{2}-2}b^{m_{1}+m_{2}+1}$. We
claim:%
\begin{equation}
b^{m_{2}}A_{1}=b^{m_{1}}B_{2},b^{m_{1}}A_{2}=b^{m_{2}}B_{1},b^{m_{2}}C_{1}=b^{m_{1}}C_{2},b\neq \pm 1.
\label{5}
\end{equation}%
It is trivial that $b^{m_{2}}C_{1}=b^{m_{1}}C_{2}$. Now $%
b^{m_{2}}A_{1}=b^{m_{1}}B_{2}\iff $

$(-1)^{m_{1}}b^{m_{1}+2m_{2}+2-p}(1-b)^{-m_{1}-m_{2}-2}(1+b)^{-m_{2}-1}\sum%
\limits_{k=0}^{m_{1}}\sum\limits_{l=0}^{m_{1}-k}\binom{m_{2}+l}{m_{2}}%
(-1)^{k}\binom{p}{k}(1-b)^{k}(1+b)^{-l}=$

$(-1)^{m_{2}}b^{m_{1}+m_{2}+p+1}(1-b)^{-m_{1}-m_{2}-2}(1+b)^{-m_{2}-1}\sum%
\limits_{k=0}^{m_{1}}\sum\limits_{l=0}^{m_{1}-k}\binom{m_{2}+l}{m_{2}}\binom{%
p}{k}(1-b)^{k}(1+b)^{-l}b^{l-k}\iff $%
\begin{equation}
\tsum\limits_{k=0}^{m_{1}}\tsum\limits_{l=0}^{m_{1}-k}\tbinom{m_{2}+l}{m_{2}}%
(-1)^{k}\tbinom{p}{k}(1-b)^{k}(1+b)^{-l}=b^{m_{1}}\tsum\limits_{k=0}^{m_{1}}%
\tsum\limits_{l=0}^{m_{1}-k}\tbinom{m_{2}+l}{m_{2}}\tbinom{p}{k}%
(1-b)^{k}(1+b)^{-l}b^{l-k},  \label{6}
\end{equation}%
and $b^{m_{1}}A_{2}=b^{m_{2}}B_{1}\iff $

$(-1)^{m_{2}}b^{m_{2}+2m_{1}+2-p}(1-b)^{-m_{1}-m_{2}-2}(1+b)^{-m_{1}-1}\sum%
\limits_{k=0}^{m_{2}}\sum\limits_{l=0}^{m_{2}-k}\binom{m_{1}+l}{m_{1}}%
(-1)^{k}\binom{p}{k}(1-b)^{k}(1+b)^{-l}=$

$(-1)^{m_{1}}b^{m_{1}+m_{2}+p+1}(1-b)^{-m_{1}-m_{2}-2}(1+b)^{-m_{1}-1}\sum%
\limits_{k=0}^{m_{2}}\sum\limits_{l=0}^{m_{2}-k}\binom{m_{1}+l}{m_{1}}\binom{%
p}{k}(1-b)^{k}(1+b)^{-l}b^{l-k}\iff $%
\begin{equation}
\tsum\limits_{k=0}^{m_{2}}\tsum\limits_{l=0}^{m_{2}-k}(-1)^{k}\tbinom{m_{1}+l%
}{m_{1}}\tbinom{p}{k}(1-b)^{k}(1+b)^{-l}=b^{m_{2}}\tsum\limits_{k=0}^{m_{2}}%
\tsum\limits_{l=0}^{m_{2}-k}\tbinom{m_{1}+l}{m_{1}}\tbinom{p}{k}%
(1-b)^{k}(1+b)^{-l}b^{l-k}.  \label{7}
\end{equation}%
(\ref{6}) is precisely Lemma \ref{L2}, and the proof of (\ref{7}) is very
similar to the proof of Lemma \ref{L2}. More simply, one can just
interchange $m_{1}$ and $m_{2}$ in Lemma \ref{L2}, since Lemma \ref{L2}
actually holds for all integers $m_{1}$ and $m_{2}$(see the remark following
Lemma \ref{L2}). That proves (\ref{5}), which immediately gives 
\begin{equation}
b^{m_{2}}\left( A_{1}+B_{1}+C_{1}\right) =b^{m_{1}}\left(
A_{2}+B_{2}+C_{2}\right) .  \label{8}
\end{equation}%
Now, if $f(x)=x^{(m_{1}+m_{2}+1)/2}$, then $M_{p,m_{1},m_{2}}$ is a
homogeneous mean. Thus it suffices to prove that $M_{p,m_{1},m_{2}}\left( 
\tfrac{1}{b},b\right) =1,b\neq 1,b\geq 0$, which is equivalent to $\left( 1-%
\tfrac{1}{b}\right) ^{m}f[1,(1/b)^{m_{1}+1},b^{m_{2}+1}]=\left( 1-b\right)
^{m}f[1,(1/b)^{m_{2}+1},b^{m_{1}+1}]$ by (\ref{h1h2}) with $a=\dfrac{1}{b}$
and $x=1$. A little simplification yields $%
b^{m_{2}}f[1,(1/b)^{m_{1}+1},b^{m_{2}+1}]=b^{m_{1}}f[1,(1/b)^{m_{2}+1},b^{m_{1}+1}]
$, which follows directly from (\ref{8}) using (\ref{4}).
\end{proof}

\begin{remark}
There are various well known integral representations for divided
differences which might be used to give a shorter proof of Theorem \ref{T5}.
This author, however, was not able to make such a proof work.
\end{remark}

Before proving our next result, we need a theorem about Cauchy Mean Values,
which have been discussed by many authors. In particular, we use results
from the paper by Leach and Sholander \cite{ls}. Let $I$ be an open interval
of real numbers and consider two given functions $f,g\in C^{n}\left(
I\right) $. Suppose that $g^{(n)}(x)\neq 0$ for $x\in I$ and that $\phi $ is
monotone on $I$, where $\phi (x)=\tfrac{f^{(n)}(x)}{g^{(n)}(x)}$. Given $n+1$
numbers $\left\{ x_{0},x_{1},...,x_{n}\right\} \subseteq I$, there is a
unique $c,\min \left\{ x_{0},x_{1},...,x_{n}\right\} \leq c\leq \max \left\{
x_{0},x_{1},...,x_{n}\right\} $, such that $\tfrac{f[x_{0},x_{1},...,x_{n}]}{%
g[x_{0},x_{1},...,x_{n}]}=\tfrac{f^{(n)}(c)}{g^{(n)}(c)}$. Of course, if the 
$x_{0},x_{1},...,x_{n}$ are not distinct, we use the extended definition of
the divided difference $f[x_{0},x_{1},...,x_{n}]$ for confluent nodes. This
defines a mean $c=M_{f,g}(x_{0},x_{1},...,x_{n})$. We state the following
result of Leach and Sholander from (\cite{ls}, Theorem 3) with the notation
altered slightly for our purposes.

\begin{theorem}
\label{T6}If $\phi ^{\prime }(x)$ is never $0$ on $I$, then $\tfrac{\partial 
}{\partial x_{k}}M_{f,g}(x_{0},x_{1},...,x_{n})>0$ for $k=0,1,...,n$.
\end{theorem}

\qquad Now we prove the following lemma.

\begin{lemma}
\label{L3}Let $I=(a,b),0<a<b$ be a given open interval, let $m_{2}<m_{1}$ be
two given nonnegative integers, with $n=m_{1}+m_{2}+1$, and suppose that $%
f,g\in C^{n+2}(I)$ with $f^{(n+1)}$ and $g^{(n+1)}$ nonzero on $I$. Assume
also that $g^{(n+1)}(x)$ and $\phi ^{\prime }(x)$ are never $0$ on $I$,
where $\phi (x)=\tfrac{f^{(n+1)}(x)}{g^{(n+1)}(x)}$. Let $\zeta _{P},\zeta
_{Q}\in I$ be the unique values satisfying $\tfrac{%
f[x,a^{m_{1}+1},b^{m_{2}+1}]}{g[x,a^{m_{1}+1},b^{m_{2}+1}]}=\tfrac{%
f^{(n+1)}(\zeta _{P})}{g^{(n+1)}(\zeta _{P})}$ and $\tfrac{%
f[x,a^{m_{2}+1},b^{m_{1}+1}]}{g[x,a^{m_{2}+1},b^{m_{1}+1}]}=\tfrac{%
f^{(n+1)}(\zeta _{Q})}{g^{(n+1)}(\zeta _{Q})}$. Then $\zeta _{P}<\zeta _{Q}$.
\end{lemma}

\begin{proof}
$\tfrac{f[x,a^{m_{1}+1},b^{m_{2}+1}]}{g[x,a^{m_{1}+1},b^{m_{2}+1}]}=\tfrac{%
f[x_{0},x_{1},...,x_{n}]}{g[x_{0},x_{1},...,x_{n}]}$ where $%
x_{0}=x,x_{1}=\cdots =x_{m_{1}+1}=a,$ and $x_{m_{1}+2}=\cdots
=x_{m_{1}+m_{2}+2}=b$, while $\tfrac{f[x,a^{m_{2}+1},b^{m_{1}+1}]}{%
g[x,a^{m_{2}+1},b^{m_{1}+1}]}=\tfrac{f[x_{0},x_{1},...,x_{n}]}{%
g[x_{0},x_{1},...,x_{n}]}$ where $x_{0}=x,x_{1}=\cdots =x_{m_{2}+1}=a,$ and $%
x_{m_{2}+2}=\cdots =x_{m_{1}+m_{2}+2}=b$. Then $\zeta _{P}=M_{f,g}\left(
x,a^{m_{1}+1},b^{m_{2}+1}\right) $ and $\zeta _{Q}=M_{f,g}\left(
x,a^{m_{2}+1},b^{m_{1}+1}\right) $, where $M_{f,g}$ denotes the mean defined
above. Since $m_{2}<m_{1}$ and $a<b$, by Theorem \ref{T6}, $\zeta _{P}<\zeta
_{Q}$.
\end{proof}

Recall that the means discussed in this paper are denoted by $%
M_{f,m_{1},m_{2}}(a,b)$, where $M_{f,m_{1},m_{2}}(a,b)$ is the unique
solution, in $(a,b)$, of the equation $E_{P}(x)=(-1)^{m_{1}-m_{2}}E_{Q}(x)$, 
$E_{P}(x)$ and $E_{Q}(x)$ given by (\ref{epq}). We now prove a result about
when $M_{f,m_{1},m_{2}}$ and $M_{g,m_{1},m_{2}}$ are comparable. For any
sufficiently smooth $f$, we let $P_{f}$ and $Q_{f}$ denote the Hermite
interpolants satisfying (\ref{PQ}). We also let $E_{P,f}=f-P_{f}$ and so on.

\begin{theorem}
\label{T7}Suppose that $\phi =\tfrac{f^{(n+1)}}{g^{(n+1)}}$ is strictly
monotonic on $(0,\infty )$, where $f,g\in C^{n+1}(0,\infty )$. Then the
means $M_{f,m_{1},m_{2}}$ and $M_{g,m_{1},m_{2}}$ are strictly comparable.
That is, either $M_{f,m_{1},m_{2}}(a,b)<M_{g,m_{1},m_{2}}(a,b)$ or $%
M_{f,m_{1},m_{2}}(a,b)>M_{g,m_{1},m_{2}}(a,b)$ for all $(a,b)\in \Re
_{2}^{+} $.
\end{theorem}

\begin{proof}
Suppose that $M_{f,m_{1},m_{2}}(a,b)=M_{g,m_{1},m_{2}}(a,b)=x_{0}$ for 
\textit{some} $(a,b)\in O=\left\{ (x,y):0<x<y\right\} $. Note that $%
g(x_{0})-P_{g}(x_{0})\neq 0$ and $g(x_{0})-Q_{g}(x_{0})\neq 0$ since $%
g^{(n+1)}$ is nonzero on $I$. Then $%
E_{P,f}(x_{0})=(-1)^{m_{1}-m_{2}}E_{Q,f}(x_{0})$ and $%
E_{P,g}(x_{0})=(-1)^{m_{1}-m_{2}}E_{Q,g}(x_{0})$, which implies that $\tfrac{%
E_{P,f}(x_{0})}{E_{P,g}(x_{0})}=\tfrac{E_{Q,f}(x_{0})}{E_{Q,g}(x_{0})}$. By (%
\ref{epq}), we then have $\tfrac{f[x_{0},a^{m_{1}+1},b^{m_{2}+1}]}{%
g[x_{0},a^{m_{1}+1},b^{m_{2}+1}]}=\tfrac{f[x_{0},a^{m_{2}+1},b^{m_{1}+1}]}{%
g[x_{0},a^{m_{2}+1},b^{m_{1}+1}]}$. Let $\zeta _{P}=\phi ^{-1}\left( \tfrac{%
f[x_{0},a^{m_{1}+1},b^{m_{2}+1}]}{g[x_{0},a^{m_{1}+1},b^{m_{2}+1}]}\right) $
and $\zeta _{Q}=\phi ^{-1}\left( \tfrac{f[x_{0},a^{m_{2}+1},b^{m_{1}+1}]}{%
g[x_{0},a^{m_{2}+1},b^{m_{1}+1}]}\right) $. By Lemma \ref{L3}, $\zeta
_{P}<\zeta _{Q}$, which contradicts the fact that $\tfrac{%
f[x_{0},a^{m_{1}+1},b^{m_{2}+1}]}{g[x_{0},a^{m_{1}+1},b^{m_{2}+1}]}=\tfrac{%
f[x_{0},a^{m_{2}+1},b^{m_{1}+1}]}{g[x_{0},a^{m_{2}+1},b^{m_{1}+1}]}$. Thus $%
M_{f,m_{1},m_{2}}(a,b)$ and $M_{g,m_{1},m_{2}}(a,b)$ are never equal on $O$.
Since $M_{f,m_{1},m_{2}}$ and $M_{g,m_{1},m_{2}}$ are each continuous on $O$
and $O$ is connected, that proves that either $%
M_{f,m_{1},m_{2}}(a,b)<M_{g,m_{1},m_{2}}(a,b)$ or $%
M_{f,m_{1},m_{2}}(a,b)>M_{g,m_{1},m_{2}}(a,b)$ for all $(a,b)\in O$ by the
intermediate value theorem. Since the means \ $M_{f,m_{1},m_{2}}$ are
symmetric, that proves Theorem \ref{T7}.
\end{proof}

\begin{theorem}
\label{T8}Let $m_{2}<m_{1}$ be two given nonnegative integers, with $%
n=m_{1}+m_{2}+1$, and suppose that $f,g\in C^{n+2}(0,\infty )$. Then $%
M_{f,m_{1},m_{2}}(a,b)=M_{g,m_{1},m_{2}}(a,b)$ for all $(a,b)\in \Re
_{2}^{+} $ if and only if $g(x)=cf(x)+p(x)$ for some constant $c$ and some
polynomial $p\in \pi _{n}$.
\end{theorem}

\begin{proof}
$(\Longleftarrow $ Suppose that $g(x)=cf(x)+p(x)$ for some constant $c$ and
some polynomial $p\in \pi _{n}$. Then it is trivial that $P_{f}=P_{g}$ and $%
Q_{f}=Q_{g}$, which implies that $%
M_{f,m_{1},m_{2}}(a,b)=M_{g,m_{1},m_{2}}(a,b)$ for all $(a,b)\in \Re
_{2}^{+} $.

$(\Longrightarrow $ Suppose that $%
M_{f,m_{1},m_{2}}(a,b)=M_{g,m_{1},m_{2}}(a,b)$ for all $(a,b)\in \Re
_{2}^{+} $, and assume that $\phi (x)=\tfrac{f^{(n+1)}(x)}{g^{(n+1)}(x)}$ is
not a constant function on $(0,\infty )$. Then $\phi $ is strictly monotone
on some open interval $I$ since $\phi ^{\prime }$ is continuous. Arguing
exactly as in the proof of Theorem \ref{T7}, with $I$ replacing $(0,\infty )$%
, we conclude that either $M_{f,m_{1},m_{2}}(a,b)<M_{g,m_{1},m_{2}}(a,b)$ or 
$M_{f,m_{1},m_{2}}(a,b)>M_{g,m_{1},m_{2}}(a,b)$ for all $(a,b)\in I$, which
is a contradiction. Thus $\tfrac{f^{(n+1)}(x)}{g^{(n+1)}(x)}$ must be a
constant function on $(0,\infty )$, which then implies that $g(x)=cf(x)+p(x)$
for some constant $c$ and some polynomial $p\in \pi _{n}$.
\end{proof}

The proof of the following theorem is very similar to the proofs of (\cite%
{h1}, lemma 1.2) and ( \cite{h1}, Theorem 1.4 and its Corollary), and we
omit them.

\begin{theorem}
\label{T9}Suppose that $f\in C^{n+2}(0,\infty )$ and that $M_{f,m_{1},m_{2}}$
is a homogeneous mean. Then $f^{(n+1)}(x)=cx^{p}$ for some real numbers $c$
and $p$.
\end{theorem}

\qquad Theorem \ref{T9} implies that the means $M_{p,m_{1},m_{2}}$ are the
only homogeneous means among the general class of means $M_{f,m_{1},m_{2}}$.

\begin{theorem}
\label{T10}$M_{p,m_{1},m_{2}}(a,b)$ is increasing in $p$ for each fixed $%
m_{1},m_{2},a,$ and $b$.
\end{theorem}

\begin{proof}
Let $f(x)=x^{p_{1}},g(x)=x^{p_{2}}$, where $p_{1}<p_{2}$. Then $\phi (x)=%
\tfrac{f^{(n+1)}(x)}{g^{(n+1)}(x)}=x^{p_{1}-p_{2}}$ is strictly monotonic on 
$(0,\infty )$. Let $0<a<b$ be fixed and let $O=\left\{ (p_{1},p_{2})\in \Re
_{2}:p_{1}<p_{2}\right\} $. By Theorem \ref{T7}, $M_{p_{1},m_{1},m_{2}}(a,b)%
\neq M_{p_{2},m_{1},m_{2}}(a,b)$ for all $(p_{1},p_{2})\in O$. Since $O$ is
connected and $M_{p,m_{1},m_{2}}(a,b)$ is a continuous function of $p$,
either $M_{p_{1},m_{1},m_{2}}(a,b)<M_{p_{2},m_{1},m_{2}}(a,b)$ or $%
M_{p_{1},m_{1},m_{2}}(a,b)>M_{p_{2},m_{1},m_{2}}(a,b)$ for all $%
(p_{1},p_{2})\in O$ by the intermediate value theorem. By Theorems \ref{T3}
and \ref{T4}, we must have $%
M_{p_{1},m_{1},m_{2}}(a,b)<M_{p_{2},m_{1},m_{2}}(a,b)$ for all $%
(p_{1},p_{2})\in O$ since it is well known that $H(a,b)\leq A(a,b)$. Since $%
a<b$ was arbitrary and $M_{p_{1},m_{1},m_{2}}$ is symmetric, that proves
Theorem \ref{T10}.
\end{proof}

The following theorem discusses the asymptotic behavior of $%
M_{p,m_{1},m_{2}} $ as $p$ approaches $\infty $ or $-\infty $.

\begin{theorem}
\label{T11}\textbf{\ }$\lim\limits_{p\rightarrow \infty
}M_{p,m_{1},m_{2}}(a,b)=\max \{a,b\}$ and $\lim\limits_{p\rightarrow -\infty
}M_{p,m_{1},m_{2}}(a,b)=\min \{a,b\}$.
\end{theorem}

\begin{proof}
Since $M_{p,m_{1},m_{2}}(a,b)$ is symmetric, we may assume that $a<b$. We
prove that $\lim\limits_{p\rightarrow \infty }M_{p,m_{1},m_{2}}(a,b)=b$, the
proof of the other case being similar. By (\ref{meandef}), (\ref{2}), and (%
\ref{3}), $M_{p,m_{1},m_{2}}(a,b)$ is the unique solution, in $(a,b)$, of
the equation $(x-a)^{m_{1}-m_{2}}\left( A_{1}+B_{1}+C_{1}\right)
=(b-x)^{m_{1}-m_{2}}\left( A_{2}+B_{2}+C_{2}\right) $, where $f(x)=x^{p}$
and $A_{j},B_{j},C_{j},j=1,2$ are given by (\ref{a1b1}), (\ref{a2b2c2}), and
(\ref{c1}). For $a\leq x\leq b$, it follows easily that $\tfrac{A_{1}}{%
\binom{p}{m_{1}}b^{p}},\tfrac{B_{1}}{\binom{p}{m_{1}}b^{p}},\tfrac{C_{1}}{%
\binom{p}{m_{1}}b^{p}},\tfrac{A_{2}}{\binom{p}{m_{1}}b^{p}},$ and $\tfrac{%
C_{2}}{\binom{p}{m_{1}}b^{p}}$ each approach $0$ as $p\rightarrow \infty $.
In the double summation for $B_{2}$, take $k=m_{1}$, which implies that $l=0$
and thus $\tfrac{B_{2}}{\binom{p}{m_{1}}b^{p}}\rightarrow
(b-a)^{-m_{2}-1}(b-x)^{-1}b^{-m_{1}}$as $p\rightarrow \infty $. Thus $%
(x-a)^{m_{1}-m_{2}}\left( A_{1}+B_{1}+C_{1}\right)
-(b-x)^{m_{1}-m_{2}}\left( A_{2}+B_{2}+C_{2}\right) \rightarrow
-(x-b)^{m_{1}-m_{2}-1}(b-a)^{-m_{2}-1}b^{-m_{1}}$ as $p\rightarrow \infty $,
which easily implies that $M_{p,m_{1},m_{2}}(a,b)$ must be approaching $b$
if $m_{1}-m_{2}>1$. We now consider the case $m_{1}=1$, $m_{2}=0$
separately. Then $M_{p,m_{1},m_{2}}(a,b)$ is the unique solution, in $(a,b)$%
, of the equation $(x-a)f[x,a,a,b]+(x-b)f[x,a,b,b]=0,f(x)=x^{p}$. Using $%
f[x,a,a,b]=\tfrac{\tfrac{f(x)-f(a)-(x-a)f\,^{\prime }(a)}{(x-a)^{2}}-\tfrac{%
f(b)-f(a)-(b-a)f\,^{\prime }(a)}{(b-a)^{2}}}{x-b}$ and $f[x,a,b,b]=\tfrac{%
\tfrac{f(x)-f(b)-(x-b)f\,^{\prime }(b)}{(x-b)^{2}}-\tfrac{%
f(a)-f(b)-(a-b)f\,^{\prime }(b)}{(b-a)^{2}}}{x-a}$, some simplification
yields the equation $L_{p}(x)=0$, where $L_{p}(x)=2\left( x^{p}-a^{p}\right)
(b-a)-2\left( b^{p}-a^{p}\right) (x-a)-p\left( b^{p-1}-a^{p-1}\right)
(x-b)(x-a)$. For $a\leq x\leq b$, $\tfrac{L_{p}(x)}{p\left(
b^{p}-a^{p}\right) }\rightarrow \tfrac{1}{b}(x-b)(x-a)$ as $p\rightarrow
\infty $. Since $M_{p,m_{1},m_{2}}$ is increasing in $p$ by Theorem \ref{T10}%
, $M_{p,m_{1},m_{2}}(a,b)$ must be approaching $b$ as $p\rightarrow \infty $.
\end{proof}

\section{Special Cases}

We now investigate the special case when $m_{1}-m_{2}=2,$ where $m_{1}+m_{2}$
is even. In this case, the mean $M_{f,m_{1},m_{2}}$ can be obtained by
solving a linear equation. In particular, if $f(x)=x^{p}$ where $p$ is an
integer, then $M_{p,m_{1},m_{2}}$ is a rational mean. Since $%
P^{(j)}(a)=Q^{(j)}(a)$ and $P^{(j)}(b)=Q^{(j)}(b),$ $j=0,1,...,m_{2}$, $P-Q$
has zeros of multiplicity $m_{2}+1$ at $x=a$ and at $x=b$. Thus $%
P(x)-Q(x)=(x-a)^{m_{2}+1}(x-b)^{m_{2}+1}R(x)$, where $R$ is a polynomial of
degree $m_{1}-m_{2}-1$. Using the formulas in \cite{aw} for Hermite
interpolation, one can directly compute the polynomials $P$ and $Q$ which
satisfy (\ref{PQ}).

\begin{gather}
P(x)=\left( \tfrac{x-b}{a-b}\right)
^{m_{2}+1}\tsum\limits_{j=0}^{m_{1}}\tsum\limits_{k=0}^{m_{1}-j}\tfrac{%
(x-a)^{j}}{j!}\tbinom{m_{2}+k}{k}\left( \tfrac{x-a}{b-a}\right)
^{k}f^{(j)}(a)+  \label{P} \\
\left( \tfrac{x-a}{b-a}\right)
^{m_{1}+1}\tsum\limits_{j=0}^{m_{2}}\tsum\limits_{k=0}^{m_{2}-j}\tfrac{%
(x-b)^{j}}{j!}\tbinom{m_{1}+k}{k}\left( \tfrac{x-b}{a-b}\right)
^{k}f^{(j)}(b)  \notag
\end{gather}

and

\begin{gather}
Q(x)=\left( \tfrac{x-b}{a-b}\right)
^{m_{1}+1}\tsum\limits_{j=0}^{m_{2}}\tsum\limits_{k=0}^{m_{2}-j}\tfrac{%
(x-a)^{j}}{j!}\tbinom{m_{1}+k}{k}\left( \tfrac{x-a}{b-a}\right)
^{k}f^{(j)}(a)+  \label{Q} \\
\left( \tfrac{x-a}{b-a}\right)
^{m_{2}+1}\tsum\limits_{j=0}^{m_{1}}\tsum\limits_{k=0}^{m_{1}-j}\tfrac{%
(x-b)^{j}}{j!}\tbinom{m_{2}+k}{k}\left( \tfrac{x-b}{a-b}\right)
^{k}f^{(j)}(b)  \notag
\end{gather}%
Since $m_{1}-m_{2}=2$, $R$ is a linear polynomial, which implies that $%
P(x)-Q(x)=(x-a)^{m_{2}+1}(x-b)^{m_{2}+1}(cx+d)$. We now determine $c$ and $d$%
. First, $d=\tfrac{P(0)-Q(0)}{a^{m_{2}+1}b^{m_{2}+1}}=\tfrac{%
E_{Q}(0)-E_{P}(0)}{a^{m_{2}+1}b^{m_{2}+1}}=\tfrac{%
a^{m_{2}+1}b^{m_{1}+1}f[0,a^{m_{2}+1},b^{m_{1}+1}]-a^{m_{1}+1}b^{m_{2}+1}f[0,a^{m_{1}+1},b^{m_{2}+1}]%
}{a^{m_{2}+1}b^{m_{2}+1}}=$

$%
b^{m_{1}-m_{2}}f[0,a^{m_{2}+1},b^{m_{1}+1}]-a^{m_{1}-m_{2}}f[0,a^{m_{1}+1},b^{m_{2}+1}]\Rightarrow 
$%
\begin{equation}
d=b^{2}f[0,a^{m_{2}+1},b^{m_{2}+3}]-a^{2}f[0,a^{m_{2}+3},b^{m_{2}+1}]
\label{9}
\end{equation}

Again, using the formula discussed earlier due to Spitzbart (see \cite{s},
Theorem 2), $f[0,a^{m_{1}+1},b^{m_{2}+1}]=$

$\sum\limits_{k=0}^{m_{1}}\sum\limits_{l=0}^{m_{1}-k}\tfrac{1}{k!}%
(-1)^{m_{1}+k}\binom{m_{2}+l}{m_{2}}%
(a-b)^{-m_{2}-1-l}a^{-m_{1}-1+k+l}f^{(k)}(a)+$

$\sum\limits_{k=0}^{m_{2}}\sum\limits_{l=0}^{m_{2}-k}\tfrac{1}{k!}%
(-1)^{m_{2}+k}\binom{m_{1}+l}{m_{1}}%
(b-a)^{-m_{1}-1-l}b^{-m_{2}-1+k+l}f^{(k)}(b)+$

$a^{-m_{1}-1}b^{-m_{2}-1}f(0)$, and

$f[0,a^{m_{2}+1},b^{m_{1}+1}]=$

$\sum\limits_{k=0}^{m_{2}}\sum\limits_{l=0}^{m_{2}-k}\tfrac{1}{k!}%
(-1)^{m_{2}+k}\binom{m_{1}+l}{m_{1}}%
(a-b)^{-m_{1}-1-l}a^{-m_{2}-1+k+l}f^{(k)}(a)+$

$\sum\limits_{k=0}^{m_{1}}\sum\limits_{l=0}^{m_{1}-k}\tfrac{1}{k!}%
(-1)^{m_{1}+k}\binom{m_{2}+l}{m_{2}}%
(b-a)^{-m_{2}-1-l}b^{-m_{1}-1+k+l}f^{(k)}(b)+$

$a^{-m_{2}-1}b^{-m_{1}-1}f(0)$. Letting $m_{1}=m_{2}+2$ gives $%
f[0,a^{m_{2}+1},b^{m_{2}+3}]=$

$\sum\limits_{k=0}^{m_{2}}\sum\limits_{l=0}^{m_{2}-k}\tfrac{1}{k!}%
(-1)^{m_{2}+k}\binom{m_{2}+l+2}{m_{2}+2}%
(a-b)^{-m_{2}-3-l}a^{-m_{2}-1+k+l}f^{(k)}(a)+$

$\sum\limits_{k=0}^{m_{2}+2}\sum\limits_{l=0}^{m_{2}+2-k}\tfrac{1}{k!}%
(-1)^{m_{2}+k}\binom{m_{2}+l}{m_{2}}%
(b-a)^{-m_{2}-1-l}b^{-m_{2}-3+k+l}f^{(k)}(b)+$

$a^{-m_{2}-1}b^{-m_{2}-3}f(0)$, and $f[0,a^{m_{2}+3},b^{m_{2}+1}]=$

$\sum\limits_{k=0}^{m_{2}+2}\sum\limits_{l=0}^{m_{2}+2-k}\tfrac{1}{k!}%
(-1)^{m_{2}+k}\binom{m_{2}+l}{m_{2}}%
(a-b)^{-m_{2}-1-l}a^{-m_{2}-3+k+l}f^{(k)}(a)+$

$\sum\limits_{k=0}^{m_{2}}\sum\limits_{l=0}^{m_{2}-k}\tfrac{1}{k!}%
(-1)^{m_{2}+k}\binom{m_{2}+2+l}{m_{2}+2}%
(b-a)^{-m_{2}-3-l}b^{-m_{2}-1+k+l}f^{(k)}(b)+$

$a^{-m_{2}-3}b^{-m_{2}-1}f(0)$. Hence, by (\ref{9}), $d=$ $%
b^{2}f[0,a^{m_{2}+1},b^{m_{2}+3}]-$

$a^{2}f[0,a^{m_{2}+3},b^{m_{2}+1}]=$

$\sum\limits_{k=0}^{m_{2}}\sum\limits_{l=0}^{m_{2}-k}\tfrac{(-1)^{m_{2}+k}}{%
k!}\binom{m_{2}+l+2}{m_{2}+2}%
b^{2}(a-b)^{-m_{2}-3-l}a^{-m_{2}-1+k+l}f^{(k)}(a)+$

$\sum\limits_{k=0}^{m_{2}+2}\sum\limits_{l=0}^{m_{2}+2-k}\tfrac{%
(-1)^{m_{2}+k}}{k!}\binom{m_{2}+l}{m_{2}}%
(b-a)^{-m_{2}-1-l}b^{-m_{2}-1+k+l}f^{(k)}(b)+$

$a^{-m_{2}-1}b^{-m_{2}-1}f(0)-\sum\limits_{k=0}^{m_{2}+2}\sum%
\limits_{l=0}^{m_{2}+2-k}\tfrac{(-1)^{m_{2}+k}}{k!}\binom{m_{2}+l}{m_{2}}%
(a-b)^{-m_{2}-1-l}a^{-m_{2}-1+k+l}f^{(k)}(a)$

$-\sum\limits_{k=0}^{m_{2}}\sum\limits_{l=0}^{m_{2}-k}\tfrac{(-1)^{m_{2}+k}}{%
k!}\binom{m_{2}+2+l}{m_{2}+2}%
a^{2}(b-a)^{-m_{2}-3-l}b^{-m_{2}-1+k+l}f^{(k)}(b)-$

$a^{-m_{2}-1}b^{-m_{2}-1}f(0)=$

$\sum\limits_{k=0}^{m_{2}}\sum\limits_{l=0}^{m_{2}-k}\tfrac{(-1)^{m_{2}+k}}{%
k!}\binom{m_{2}+l+2}{m_{2}+2}(a-b)^{-m_{2}-3-l}\left(
b^{2}a^{-m_{2}-1+k+l}f^{(k)}(a)+(-1)^{m_{2}+l}a^{2}b^{-m_{2}-1+k+l}f^{(k)}(b)\right) + 
$

$\sum\limits_{k=0}^{m_{2}+2}\sum\limits_{l=0}^{m_{2}+2-k}\tfrac{%
(-1)^{m_{2}+k}}{k!}\binom{m_{2}+l}{m_{2}}(b-a)^{-m_{2}-1-l}\left(
b^{-m_{2}-1+k+l}f^{(k)}(b)+(-1)^{m_{2}+l}a^{-m_{2}-1+k+l}f^{(k)}(a)\right) $%
. Now we find $c$. It is not hard to show, using (\ref{PQ}), that the
coefficient, $c_{P,m_{1},m_{2}\text{,}}$ of the highest power in $P$, which
is $x^{m_{1}+m_{2}+1}$, is given by $\sum\limits_{j=0}^{m_{1}}\tfrac{\binom{%
m_{2}+m_{1}-j}{m_{2}}f^{(j)}(a)}{j!(a-b)^{m_{2}+1}(b-a)^{m_{1}-j}}%
+\sum\limits_{j=0}^{m_{2}}\tfrac{\binom{m_{2}+m_{1}-j}{m_{1}}f^{(j)}(b)}{%
j!(b-a)^{m_{1}+1}(a-b)^{m_{2}-j}}$ or

\begin{equation}
c_{P,m_{1},m_{2}}=\tfrac{(-1)^{m_{2}}}{(b-a)^{m_{1}+m_{2}+1}}\left(
\tsum\limits_{j=0}^{m_{2}}\tfrac{(-1)^{j}\binom{m_{2}+m_{1}-j}{m_{1}}%
(b-a)^{j}f^{(j)}(b)}{j!}-\tsum\limits_{j=0}^{m_{1}}\tfrac{\binom{%
m_{2}+m_{1}-j}{m_{2}}(b-a)^{j}f^{(j)}(a)}{j!}\right)  \label{cp}
\end{equation}%
Similarly, the coefficient, $c_{Q,m_{1},m_{2}\text{,}}$ of the highest power
in $Q$, which is $x^{m_{1}+m_{2}+1}$, is given by $\sum\limits_{j=0}^{m_{2}}%
\tfrac{\binom{m_{2}+m_{1}-j}{m_{1}}f^{(j)}(a)}{%
j!(a-b)^{m_{1}+1}(b-a)^{m_{2}-j}}+\sum\limits_{j=0}^{m_{1}}\tfrac{\binom{%
m_{2}+m_{1}-j}{m_{2}}f^{(j)}(b)}{j!(b-a)^{m_{2}+1}(a-b)^{m_{1}-j}}$ or

\begin{equation}
c_{Q,m_{1},m_{2}}=\tfrac{(-1)^{m_{1}}}{(b-a)^{m_{1}+m_{2}+1}}\left(
\tsum\limits_{j=0}^{m_{1}}\tfrac{(-1)^{j}\binom{m_{2}+m_{1}-j}{m_{2}}%
(b-a)^{j}f^{(j)}(b)}{j!}-\tsum\limits_{j=0}^{m_{2}}\tfrac{\binom{%
m_{2}+m_{1}-j}{m_{1}}(b-a)^{j}f^{(j)}(a)}{j!}\right)  \label{cq}
\end{equation}%
Hence $c=\tfrac{(-1)^{m_{2}+1}}{(b-a)^{m_{1}+m_{2}+1}}\left(
\sum\limits_{j=0}^{m_{1}}\tfrac{\binom{m_{2}+m_{1}-j}{m_{2}}%
(b-a)^{j}f^{(j)}(a)}{j!}-\sum\limits_{j=0}^{m_{2}}\tfrac{(-1)^{j}\binom{%
m_{2}+m_{1}-j}{m_{1}}(b-a)^{j}f^{(j)}(b)}{j!}\right) -$

$\tfrac{(-1)^{m_{1}+1}}{(b-a)^{m_{1}+m_{2}+1}}\left(
\sum\limits_{j=0}^{m_{2}}\tfrac{\binom{m_{2}+m_{1}-j}{m_{1}}%
(b-a)^{j}f^{(j)}(a)}{j!}-\sum\limits_{j=0}^{m_{1}}\tfrac{(-1)^{j}\binom{%
m_{2}+m_{1}-j}{m_{2}}(b-a)^{j}f^{(j)}(b)}{j!}\right) \Rightarrow $

\begin{gather}
c=\tfrac{1}{(b-a)^{m_{1}+m_{2}+1}}\tsum\limits_{j=0}^{m_{1}}\tfrac{\binom{%
m_{2}+m_{1}-j}{m_{2}}%
(b-a)^{j}((-1)^{m_{2}+1}f^{(j)}(a)-(-1)^{m_{1}+j}f^{(j)}(b))}{j!}+  \label{c}
\\
\tfrac{1}{(b-a)^{m_{1}+m_{2}+1}}\tsum\limits_{j=0}^{m_{2}}\tfrac{\binom{%
m_{2}+m_{1}-j}{m_{1}}%
(b-a)^{j}((-1)^{m_{1}}f^{(j)}(a)+(-1)^{m_{2}+j}f^{(j)}(b))}{j!}.  \notag
\end{gather}

Using $m_{1}=m_{2}+2$ gives 
\begin{equation*}
M_{f,m_{1},m_{2}}(a,b)=-\tfrac{d}{c},
\end{equation*}%
where 
\begin{eqnarray}
d &=&\tsum\limits_{k=0}^{m_{2}}\tsum\limits_{l=0}^{m_{2}-k}\tfrac{1}{k!}%
(-1)^{m_{2}+k}\tbinom{m_{2}+l+2}{m_{2}+2}(a-b)^{-m_{2}-3-l}\times  \notag \\
&&\left(
b^{2}a^{-m_{2}-1+k+l}f^{(k)}(a)+(-1)^{m_{2}+l}a^{2}b^{-m_{2}-1+k+l}f^{(k)}(b)\right) +
\notag \\
&&\tsum\limits_{k=0}^{m_{2}+2}\tsum\limits_{l=0}^{m_{2}+2-k}\tfrac{1}{k!}%
(-1)^{m_{2}+k}\tbinom{m_{2}+l}{m_{2}}(b-a)^{-m_{2}-1-l}\times  \label{10} \\
&&\left(
b^{-m_{2}-1+k+l}f^{(k)}(b)+(-1)^{m_{2}+l}a^{-m_{2}-1+k+l}f^{(k)}(a)\right) 
\notag
\end{eqnarray}

and

\begin{gather}
c=\tfrac{(-1)^{m_{2}}}{(b-a)^{2m_{2}+3}}\tsum\limits_{j=0}^{m_{2}+2}\tfrac{%
\binom{2m_{2}+2-j}{m_{2}}(b-a)^{j}((-1)^{j+1}f^{(j)}(b)-f^{(j)}(a))}{j!}+
\label{11} \\
\tfrac{(-1)^{m_{2}}}{(b-a)^{2m_{2}+3}}\tsum\limits_{j=0}^{m_{2}}\tfrac{%
\binom{2m_{2}+2-j}{m_{2}+2}(b-a)^{j}((-1)^{j}f^{(j)}(b)+f^{(j)}(a))}{j!} 
\notag
\end{gather}

We now examine three special cases. For $m_{1}=4$ and $m_{2}=2$, using (\ref%
{10}) and (\ref{11}), we have $24\left( b-a\right)
^{6}d=840(f(b)-f(a))+120(3a-4b)f\,^{\prime }(b)+120(4a-3b)f\,^{\prime }(a)+$

$60\left( b-a\right) (\left( 2b-a\right) f\,^{\prime \prime }(b)-\left(
b-2a\right) f\,^{\prime \prime }(a))-4\left( b-a\right) ^{2}(\left(
b-4a\right) f\,^{\prime \prime \prime }(a)+$

$\left( 4b-a\right) f\,^{\prime \prime \prime }(b))+\left( b-a\right)
^{3}(af^{\prime \prime \prime \prime }(a)+bf^{\prime \prime \prime \prime
}(b)))$ and

$-24\left( b-a\right) ^{6}c=-120(f\,^{\prime }(b)-f\,^{\prime
}(a))+60(b-a)(f\,^{\prime \prime }(b)+f\,^{\prime \prime }(a))-$

$12\left( b-a\right) ^{2}(f\,^{\prime \prime \prime }(b)-f\,^{\prime \prime
\prime }(a))+\left( b-a\right) ^{3}(f^{\prime \prime \prime \prime
}(b)+f^{\prime \prime \prime \prime }(a))$. Thus $M_{f,4,2}(a,b)=$

$\tfrac{840(f(b)-f(a))+120(3a-4b)f\,^{\prime }(b)+120(4a-3b)f\,^{\prime
}(a)+60\left( b-a\right) (\left( 2b-a\right) f\,^{\prime \prime }(b)-\left(
b-2a\right) f\,^{\prime \prime }(a))-4\left( b-a\right) ^{2}(\left(
b-4a\right) f\,^{\prime \prime \prime }(a)+\left( 4b-a\right) f\,^{\prime
\prime \prime }(b))+\left( b-a\right) ^{3}(af^{\prime \prime \prime \prime
}(a)+bf^{\prime \prime \prime \prime }(b))\allowbreak \allowbreak
\allowbreak }{-120(f\,^{\prime }(b)-f\,^{\prime }(a))+60(b-a)(f\,^{\prime
\prime }(b)+f\,^{\prime \prime }(a))-12\left( b-a\right) ^{2}(f\,^{\prime
\prime \prime }(b)-f\,^{\prime \prime \prime }(a))+\left( b-a\right)
^{3}(f^{\prime \prime \prime \prime }(b)+f^{\prime \prime \prime \prime }(a))%
}\allowbreak \allowbreak \allowbreak \allowbreak \allowbreak \allowbreak $

For $m_{1}=3$ and $m_{2}=1$, again using (\ref{10}) and (\ref{11}), we have $%
6\left( b-a\right) ^{-4}d=-60(f(b)-f(a))+12(3b-2a)f\,^{\prime
}(b)-12(3a-2b)f\,^{\prime }(a)+3\left( b-a\right) (\left( a-3b\right)
f\,^{\prime \prime }(b)+$

$\left( b-3a\right) f\,^{\prime \prime }(a))+\left( b-a\right)
^{2}(bf\,^{\prime \prime \prime }(b)-af\,^{\prime \prime \prime }(a))$ and $%
-6\left( b-a\right) ^{4}c=-12(f\,^{\prime }(b)-f\,^{\prime
}(a))+6(b-a)(f\,^{\prime \prime }(b)+f\,^{\prime \prime }(a))-\left(
b-a\right) ^{2}(f\,^{\prime \prime \prime }(b)-f\,^{\prime \prime \prime
}(a))$. Thus

$M_{f,3,1}(a,b)=$

$\tfrac{-60(f(b)-f(a))+12(3b-2a)f\,^{\prime }(b)-12(3a-2b)f\,^{\prime
}(a)+3\left( b-a\right) (\left( a-3b\right) f\,^{\prime \prime }(b)+\left(
b-3a\right) f\,^{\prime \prime }(a))+\left( b-a\right) ^{2}(bf\,^{\prime
\prime \prime }(b)-af\,^{\prime \prime \prime }(a))}{12(f\,^{\prime
}(b)-f\,^{\prime }(a))-6(b-a)(f\,^{\prime \prime }(b)+f\,^{\prime \prime
}(a))+\left( b-a\right) ^{2}(f\,^{\prime \prime \prime }(b)-f\,^{\prime
\prime \prime }(a))}$

For $m_{1}=2$ and $m_{2}=0$ we have $d=\tfrac{(b^{2}-ab)f\,^{\prime \prime
}(b)+(ab-a^{2})f\,^{\prime \prime }(a)+(2a-4b)f\,^{\prime
}(b)+(4a-2b)f\,^{\prime }(a)+6(f(b)-f(a))}{2\left( b-a\right) ^{2}}$ and $c=-%
\tfrac{(b-a)(f\,^{\prime \prime }(b)+f\,^{\prime \prime }(a))-2(f\,^{\prime
}(b)-f\,^{\prime }(a))}{2\left( b-a\right) ^{2}}$. Thus

\begin{equation*}
M_{f,2,0}(a,b)=\tfrac{(b-a)(bf\,^{\prime \prime }(b)+af\,^{\prime \prime
}(a))+2(a-2b)f\,^{\prime }(b)+2(2a-b)f\,^{\prime }(a)+6(f(b)-f(a))}{%
(b-a)(f\,^{\prime \prime }(b)+f\,^{\prime \prime }(a))-2(f\,^{\prime
}(b)-f\,^{\prime }(a))}
\end{equation*}%
If $f(x)=x^{p}$, then some simplification yields $M_{f,2,0}(a,b)=$

$\tfrac{1}{p}\tfrac{%
b^{p-2}(p(p-5)b^{2}+p(3-p)ab+6b^{2})-a^{p-2}(p(p-5)a^{2}+p(3-p)ab+6a^{2})}{%
b^{p-2}((p-1)(b-a)-2b)+a^{p-2}((p-1)(b-a)+2a)},p\notin \{0,1,2,3\}$. The
omitted cases for $p$ can be obtained as limiting values, or one can just
let $f(x)=x^{p}\log x$ for $p\in \{0,1,2,3\}$. That yields $M_{\log
x,2,0}(a,b)=3ab\tfrac{b^{2}-a^{2}-2ab(\ln b-\ln a)}{(b-a)^{3}}=3ab\tfrac{%
b^{2}-a^{2}-2ab\ln \left( \tfrac{b}{a}\right) }{(b-a)^{3}}$, $M_{x\log
x,2,0}(a,b)=2ab\allowbreak \tfrac{(a+b)\ln \left( \tfrac{b}{a}\right) -2(b-a)%
}{b^{2}-a^{2}-2ab\ln \left( \tfrac{b}{a}\right) }$, $M_{x^{2}\log
x,2,0}(a,b)=\tfrac{1}{2}\tfrac{b^{2}-a^{2}-2ab\ln \left( \tfrac{b}{a}\right) 
}{(a+b)\ln \left( \tfrac{b}{a}\right) -2(b-a)}$, and $M_{x^{3}\log
x,2,0}(a,b)=\tfrac{1}{3}\tfrac{(b-a)^{3}}{b^{2}-a^{2}-2ab\ln \left( \tfrac{b%
}{a}\right) }$.

Finally, we consider the case $m_{1}=1$ and $m_{2}=0$, so that $m_{1}+m_{2}$
is odd. As noted earlier, $M_{p,m_{1},m_{2}}(a,b)$ is the unique solution,
in $(a,b)$, of the equation $2\left( x^{p}-a^{p}\right) (b-a)-2\left(
b^{p}-a^{p}\right) (x-a)-p\left( b^{p-1}-a^{p-1}\right) (x-b)(x-a)=0$. For $%
p=4,$ after dividing thru by $2\left( x-a\right) \left( b-a\right) \left(
b-x\right) $, we have $2\left( x-a\right) \left( b-a\right) \left(
b-x\right) \left( b^{2}-xb+ab-xa+a^{2}-x^{2}\right) =0$. This can be solved
exactly to obtain $M_{4,1,0}(a,b)=\tfrac{1}{2}\sqrt{5b^{2}+6ab+5a^{2}}-%
\tfrac{a+b}{2}$. For $p=5,$ after dividing thru by $\left( x-a\right) \left(
b-a\right) \left( b-x\right) $, we have $%
2x^{3}+2bx^{2}+2ax^{2}+2b^{2}x+2xab+2xa^{2}-3a^{3}-3b^{2}a-3ba^{2}-3b^{3}=0$%
. The root in $(a,b)$is given by $M_{5,1,0}(a,b)=\tfrac{1}{6}\sqrt[3]{%
s(a,b)+6\sqrt{t(a,b)}}-\tfrac{2}{3}\tfrac{2b^{2}+ab+2a^{2}}{\sqrt[3]{s(a,b)+6%
\sqrt{t(a,b)}}}-\tfrac{a+b}{3}$, where $s(a,b)=10\left( a+b\right) \left(
19a^{2}+2ab+19b^{2}\right) $ and $%
t(a,b)=1017b^{6}+2238b^{5}a+3495b^{4}a^{2}+4500b^{3}a^{3}+3495b^{2}a^{4}+2238a^{5}b+1017a^{6} 
$

\section{Comparisons with Taylor polynomial means}

As noted earlier, the means defined in this paper are similar to a class of
means defined in \cite{h1}, which were based on intersections of Taylor
polynomials. For $f\in C^{r+1}(I),I=(a,b),$ let $P_{c}$ denote the Taylor
polynomial to $f$ of order $r$ at $x=c$, where $r$ is an odd positive
integer. In \cite{h1} it was proved that if $f^{(r+1)}(x)\neq 0$ on $[a,b]$,
then there is a unique $u,a<u<b,$ such that $P_{a}(u)=P_{b}(u)$. This
defines a mean $m(a,b)\equiv u$, which we denote by $M_{f}^{r}(a,b)$. The
arithmetic, geometric, and harmonic means arise for both classes of means.
We now show that there are means defined in this paper which do not occur as
intersections of Taylor polynomials. In particular, consider the mean $%
M_{\log x,2,0}(a,b)=3ab\tfrac{b^{2}-a^{2}-2ab(\ln b-\ln a)}{(b-a)^{3}}$
discussed earlier. Then $h(b)=M_{\log x,2,0}(1,b)=3b\tfrac{b^{2}-1-2b\ln b}{%
(b-1)^{3}}$, $\lim\limits_{b\rightarrow 1}h^{\prime }(b)=\allowbreak \tfrac{1%
}{2}$, $\lim\limits_{b\rightarrow 1}h^{\prime \prime }(b)=\allowbreak -%
\tfrac{2}{5}$, $\lim\limits_{b\rightarrow 1}h^{\prime \prime \prime
}(b)=\allowbreak \tfrac{3}{5}$, and $\lim\limits_{b\rightarrow 1}h^{\prime
\prime \prime \prime }(b)=\allowbreak -\tfrac{48}{35}$. Since $M_{\log x,2,0}
$ is a homogeneous mean, if $M_{\log x,2,0}=M_{f}^{r}$ for some $f$, then we
may assume that $f(x)=x^{p}$ for some real number $p$ by [\cite{h1}, Theorem
1.4]. Let $k(b)=M_{p}^{r}(1,b)=M_{f}^{r}(1,b)$, where $f(x)=x^{p}$. From [%
\cite{h1}, Theorem 4.1], $k^{\prime \prime }(1)=\tfrac{p-r-1}{2(r+2)}$, $%
k^{\prime \prime \prime }(1)=\tfrac{-3(p-r-1)}{4(r+2)}$, and $k^{\prime
\prime \prime \prime }(1)=\tfrac{p-r-1}{8(r+2)^{3}(r+4)}%
(12r^{3}+8(p+13)r^{2}-4(p^{2}-12p-73)r-16(2p^{2}-p-15)$. Setting $\tfrac{%
p-r-1}{2(r+2)}=-\tfrac{2}{5}$ and $\tfrac{-3(p-r-1)}{4(r+2)}=\tfrac{3}{5}$%
implies that $r=5p+3$. Substituting into $k^{\prime \prime \prime \prime }(1)
$ gives $\allowbreak -\tfrac{12}{125}\tfrac{70p+99}{5p+7}$. Setting $-\tfrac{%
12}{125}\tfrac{70p+99}{5p+7}=\allowbreak -\tfrac{48}{35}$ implies that $p=-%
\tfrac{7}{10}$. Then $r=5\left( -\tfrac{7}{10}\right) +3=\allowbreak -\tfrac{%
1}{2}$, which is not a positive integer. Thus $M_{\log x,2,0}$ \textbf{cannot%
} occur as one of the means $M_{f}^{r}$.

\section{Open Questions and Future Research}

In \cite{h2} it was shown that $\lim\limits_{r\rightarrow \infty
}M_{p}^{r}(a,b)=H(a,b)=\tfrac{2ab}{a+b}$, where $M_{f}^{r}$ are the Taylor
polynomial means defined above. There is strong evidence that a similar
result holds for the means defined in this paper. That is,

\begin{conjecture}
$\lim\limits_{n\rightarrow \infty }M_{p,m_{1},m_{2}}(a,b)$ $=H(a,b)$, where $%
n=m_{1}+m_{2}+1$.
\end{conjecture}

More generally, analyze the asymptotic behavior of $M_{f,m_{1},m_{2}}$ as $%
n\rightarrow \infty $. As in \cite{h2}, it should follow that the arithmetic
mean arises as $\lim\limits_{n\rightarrow \infty }M_{f,m_{1},m_{2}}$. It is
then natural to ask:

\textbf{Question:} Are the arithmetic and harmonic means the only means
which arise as $\lim\limits_{n\rightarrow \infty }M_{f,m_{1},m_{2}}$ ?

We showed in Theorem \ref{T4} that $M_{-1,m_{1},m_{2}}(a,b)=H(a,b)=\tfrac{2ab%
}{a+b}$ for any $m_{1}$ and $m_{2}$. Thus for $f(x)=\tfrac{1}{x}$, $%
M_{f,m_{1},m_{2}}$ is independent of $m_{1}$ and $m_{2}$.

\begin{conjecture}
Show that the only function, $f$, for which $M_{f,m_{1},m_{2}}$ is
independent of $m_{1}$ and $m_{2}$ is $f(x)=\tfrac{C}{x}$.
\end{conjecture}

\end{document}